\documentclass[english,12pt]{article}
\usepackage[T1]{fontenc}
\usepackage{geometry}
\usepackage{url}
\usepackage{soul}
\usepackage{enumitem}
\usepackage{amsmath}
\usepackage{amsthm}
\usepackage{amssymb}

\usepackage[usenames,dvipsnames]{color}

\makeatletter
 \theoremstyle{definition}
 \newtheorem*{defn*}{\protect\definitionname}
 \theoremstyle{plain}
 \newtheorem{thm}{\protect\theoremname}[section]
  \theoremstyle{remark}
  \newtheorem*{rem*}{\protect\remarkname}
  \theoremstyle{plain}
  \newtheorem{lem}[thm]{\protect\lemmaname}
  \theoremstyle{plain}
  \newtheorem{claim}[thm]{\protect\claimname}
 \newlist{casenv}{enumerate}{4}
 \setlist[casenv]{leftmargin=*,align=left,widest={iiii}}
 \setlist[casenv,1]{label={{\itshape\ \casename} \arabic*.},ref=\arabic*}
 \setlist[casenv,2]{label={{\itshape\ \casename} \roman*.},ref=\roman*}
 \setlist[casenv,3]{label={{\itshape\ \casename\ \alph*.}},ref=\alph*}
 \setlist[casenv,4]{label={{\itshape\ \casename} \arabic*.},ref=\arabic*}
  \theoremstyle{plain}
  \newtheorem{cor}[thm]{\protect\corollaryname}
  \theoremstyle{plain}
  \newtheorem{prop}[thm]{\protect\propositionname}
  \theoremstyle{definition}
  \newtheorem{obs}[thm]{\protect\observationname}

\makeatother

\usepackage{babel}
  \providecommand{\claimname}{Claim}
  \providecommand{\corollaryname}{Corollary}
  \providecommand{\definitionname}{Definition}
  \providecommand{\lemmaname}{Lemma}
  \providecommand{\propositionname}{Proposition}
  \providecommand{\observationname}{Observation}
  \providecommand{\remarkname}{Remark}
 \providecommand{\casename}{Case}
\providecommand{\theoremname}{Theorem}

\begin{document}
\global\long\def\prob{\mathrm{Pr}}
\global\long\def\E{\mathbb{E}}
\global\long\def\var{\mathrm{Var}}
\global\long\def\Poisson{\mathrm{Poisson}}
\global\long\def\Bin{\mathrm{Bin}}
\global\long\def\mm{\hat{m}}
\global\long\def\nn{\hat{n}}
\global\long\def\hd{\hat{d}}
\global\long\def\dd{\vec{d}}
\global\long\def\ddn{\vec{d}_{(\nn)}}
\global\long\def\ddt{\vec{d}_{(t)}}
\global\long\def\hh{\vec{h}}
\global\long\def\eps{\varepsilon}
\global\long\def\xs{\mathrm{exc}}
\global\long\def\type{\mathrm{type}}
\global\long\def\dns{\mathrm{dense}}
\global\long\def\sprs{\mathrm{sparse}}
\global\long\def\cI{\mathcal{\mathcal{I}}}
\global\long\def\gnp{\mathcal{G}\left(n,p\right)}
\global\long\def\gnm{\mathcal{G}\left(n,m\right)}
\global\long\def\gnc{\mathcal{G}\left(n,c/n\right)}
\global\long\def\ggnp{G\sim\gnp}
\global\long\def\ggnc{G\sim\gnc}
\global\long\def\K{\mathcal{K}}
\global\long\def\D{\mathcal{D}}
\global\long\def\P{\mathcal{P}}

\title{Component Games on Random Graphs}

\author{%
Rani Hod\thanks{School of Computer Science, Raymond and Beverly Sackler Faculty of Exact Sciences, and Iby and Aladar Fleischman Faculty of Engineering, Tel Aviv University, Tel Aviv 6997801, Israel.
Email: \protect\url{ranihod@tau.ac.il}. Research supported by Len Blavatnik and the Blavatnik Family foundation.}
\and
Michael Krivelevich\thanks{School of Mathematical Sciences, Raymond and Beverly Sackler Faculty of Exact Sciences, Tel Aviv University, Tel Aviv 6997801, Israel.
Email: \protect\url{krivelev@tau.ac.il}.
Research supported in part by USA-Israel BSF grants~2014361 and~2018267, and by ISF grant~1261/17.}
\and
Tobias M\"{u}ller\thanks{Bernoulli Institute, Groningen University, PO Box 407, 9700 AK Groningen, The Netherlands.
Email: \protect\url{tobias.muller@rug.nl}.
Research partially supported by NWO grants 639.031.829, 639.032.529 and 612.001.409.}
\and
Alon Naor\thanks{School of Mathematical Sciences, Raymond and Beverly Sackler Faculty of Exact Sciences, Tel Aviv University, Tel Aviv 6997801, Israel.
Email: \protect\url{alonnaor@tau.ac.il}}
\and
Nicholas Wormald\thanks{School of Mathematics, Monash University VIC 3800, Australia.
Email: \protect\url{nicholas.wormald@monash.edu}.
Research supported by the Australian Laureate Fellowships grant FL120100125.}}
\maketitle
\begin{abstract}
In the $\left(1:b\right)$ component game played on a graph $G$,
two players, \textsc{Maker} and \textsc{Breaker}, alternately claim~$1$
and~$b$ previously unclaimed edges of $G$, respectively. \textsc{Maker}'s
aim is to maximise the size of a largest connected component in
her graph, while \textsc{Breaker} is trying to minimise it. We show
that the outcome of the game on the binomial
random graph is strongly correlated with the appearance of a nonempty
$(b+2)$-core in the graph.

For any integer $k$, the $k$-core of a graph is its largest subgraph
of minimum degree at least $k$. Pittel, Spencer and Wormald showed
in 1996 that for any $k\ge3$ there exists an explicitly defined constant $c_{k}$ such
that $p=c_{k}/n$ is the threshold function for the appearance of
the $k$-core in $\ggnp$. More precisely, $\ggnc$ has WHP a linear-size
$k$-core when the constant $c>c_{k}$, and an empty $k$-core when $c<c_{k}$.

We show that for any positive constant $b$, when playing the $(1:b)$
component game on $\ggnc$, \textsc{Maker} can WHP build a linear-size
component if $c>c_{b+2}$, while \textsc{Breaker} can WHP prevent
\textsc{Maker} from building larger than polylogarithmic-size components
if $c<c_{b+2}$.

For \textsc{Breaker}'s strategy, we prove a theorem which may be of
independent interest. The standard algorithm for computing the $k$-core
of any graph is to repeatedly delete (``peel'') all vertices of degree
less than $k$, as long as such vertices remain. When $\ggnc$ for
$c<c_{k}$, it was shown by Jiang, Mitzenmacher and Thaler that $\log_{k-1}\log n+\Theta(1)$
peeling iterations are WHP necessary and sufficient to obtain the
(empty) $k$-core of~$G$. Our theorem states that already after
a constant number of iterations, $G$ is WHP shattered into pieces
of polylogarithmic size.
\end{abstract}

\section{\label{sec:intro}Introduction}

Let $X$ be a finite set, let $\mathcal{F}\subseteq2^{X}$ be a family
of subsets of $X$, and let $b$ be a positive integer. In the $\left(1:b\right)$
\textsc{Maker}\textendash \textsc{Breaker} game $\left(X,\mathcal{F}\right)$,
two players, called \textsc{Maker} and \textsc{Breaker}, take turns
in claiming previously unclaimed elements of $X$. On \textsc{Maker}'s
move, she claims one element of $X$, and on \textsc{Breaker}'s move,
he claims~$b$ elements (if less than $b$ elements remain before
\textsc{Breaker's} last move, he claims all of them). The game ends
when all of the elements have been claimed by either of the players.
\textsc{Maker} wins the game $\left(X,\mathcal{F}\right)$ if by the end of
the game she has claimed all the elements of some $F\in\mathcal{F}$;
otherwise \textsc{Breaker} wins. The description of the game is completed by stating which of the players
is the first to move, though usually it makes no real difference. For convenience, we typically assume
that $\mathcal{F}$ is closed upwards, and specify only the inclusion-minimal
elements of $\mathcal{F}$.  Since these are finite, perfect information
games with no possibility of draw, for each setup of $\mathcal{F},b$
and the choice of the identity of the first player, one of the players has a strategy
to win regardless of the other player's strategy. Therefore, for a
given game we may say that the game is \textsc{Maker}\textquoteright s
win, or alternatively that it is \textsc{Breaker}\textquoteright s
win. The set $X$ is referred to as the \emph{board} of the game,
and the elements of $\mathcal{F}$ are referred to as the \emph{winning
sets}.

When $b=1$, we say that the game is \emph{unbiased}; otherwise it
is \emph{biased}, and $b$ is called the \emph{bias} of \textsc{Breaker}.
It is easy to see that \textsc{Maker}\textendash \textsc{Breaker}
games are \emph{bias monotone}. That is, if \textsc{Maker} wins some
game with bias $(1:b)$, she also wins this game with bias $(1:b')$
for every $b'\leq b$. Similarly, if \textsc{Breaker} wins a game
with bias $(1:b)$, he also wins this game with bias $(1:b')$ for
every $b'\geq b$. This bias monotonicity allows us to define the
\emph{threshold bias}: for a given game $\mathcal{F}$, the threshold
bias $b^{*}$ is the value for which \textsc{Breaker} wins the game
$\mathcal{F}$ with bias $(1:b)$ if and only if $b>b^{*}$. It is
quite easy to observe that it is never a disadvantage in a \textsc{Maker}\textendash \textsc{Breaker}
game to be the first player, and that if a player has a winning strategy
as the second player, essentially the same strategy can be used to
also win the game as the first player. Hence, when we describe a strategy
for \textsc{Maker} we assume that she is the second player, implying
that under the conditions described she can win as either a first
or a second player. The same goes for \textsc{Breaker}'s strategy.

\medskip{}

In this paper, our attention is dedicated to the $\left(1:b\right)$
\textsc{Maker}\textendash \textsc{Breaker} $s$-component game on
the binomial random graph $\mathcal{G}\left(n,p\right)$,
in which each of the ${n \choose 2}$ possible edges appears independently
with probability $p = p(n)$; that is, the board is the edge set of $\ggnp$
and the (inclusion-minimal) winning sets are the trees of $G$ with $s$
vertices. Since the board is random, our results hold with high probability
(WHP), i.e., with probability tending to $1$ as $n$ tends to infinity.

\medskip{}

For more on \textsc{Maker}\textendash \textsc{Breaker} games as well as other positional games, please see the books by Beck~\cite{B-08} and by Hefetz et al.~\cite{HKSS-14}.

\subsection{Previous results}

A natural case to consider is $s=n$; that is, the winning sets are
the spanning trees of the graph the game is played on. This $\left(1:b\right)$
$n$-component game is known as the \emph{connectivity} game.

The unbiased game was completely solved by Lehman~\cite{Lehman-64},
who showed that \textsc{Maker}, as a second player, wins the $\left(1:1\right)$
connectivity game on a graph $G$ if and only if $G$ contains two
edge-disjoint spanning trees. It follows easily from~\cite{Nash-61,Tutte-61}
that if $G$ is $2k$-edge-connected then it contains $k$ pairwise
independent spanning trees; thus, \textsc{Maker} wins the $\left(1:1\right)$
connectivity game on $4$-regular $4$-edge-connected graphs, whereas
\textsc{Breaker} trivially wins the $\left(1:1\right)$ connectivity
game on graphs with less than $2n-2$ edges, i.e., average degree under 4.
For denser graphs, since \textsc{Maker} wins the unbiased game by
such a large margin, it only seems fair to even out the odds by strengthening
\textsc{Breaker}, giving him a bias $b\ge2$. The first and most natural
board to consider is the edge set of the complete graph $K_{n}$.
Chv\'{a}tal and Erd\H{o}s~\cite{CH-78} showed that $\left(\frac{1}{4}-o\left(1\right)\right)n/\log n\le b^{*}\left(K_{n}\right)\le\left(1+o\left(1\right)\right)n/\log n$;
the upper bound was proved to be tight by Gebauer and Szab\'{o}~\cite{GS-09};
that is, \textsc{$b^{*}\left(K_{n}\right)=\left(1+o\left(1\right)\right)n/\log n$.}

Returning to $\mathcal{G}\left(n,p\right)$, Stojakovi\'{c} and Szab\'{o}~\cite{SS-05}
showed that WHP $b^{*}\left(\mathcal{G}\left(n,p\right)\right)=\Theta\left(np/\log n\right)$,
where \textsc{Breaker}'s win holds for any $0\le p\le1$, while \textsc{Maker}'s
win requires large enough $p$ (\textsc{Maker} cannot win for small
$p$ since $\ggnp$ is WHP disconnected). This was improved by Ferber
et al.~in~\cite{FGKN-15}, who showed for $p=\omega\left(\log n/n\right)$
that WHP $b^{*}\left(\mathcal{G}\left(n,p\right)\right)=\left(1+o\left(1\right)\right)np/\log n$.

A different random graph model, the random $d$-regular graph $\mathcal{G}\left(n,d\right)$
on $n$ vertices, was considered by Hefetz et al.~in~\cite{HKSS-11}.
They showed that WHP $b^{*}\left(\mathcal{G}\left(n,d\right)\right)\ge\left(1-\epsilon\right)d/\log_{2}n$
for $d=o\left(\sqrt{n}\right)$. Note that when $d=\Omega\left(\sqrt{n}\right)$,
the model $\mathcal{G}\left(n,d\right)$ is quite close to $\mathcal{G}\left(n,p\right)$
for $p=d/n$, since for this value of $p$ all degrees in $\ggnp$ are WHP $(1 + o(1))d$. Moreover,
they showed that $b^{*}\left(G\right)\le\max\left\{ 2,\bar{d}/\log n\right\} $
for any graph $G$ of average degree $\bar{d}$, so the result is
asymptotically tight.

\medskip{}
\textsc{Breaker}'s strategy in practically all results mentioned above
is to deny connectivity by isolating a single vertex. Much less is
known, however, for the case $s<n$. It seems that even if \textsc{Breaker}
is able to isolate a vertex in a constant number of moves, it does
little to prevent \textsc{Maker} from winning the $s$-component game
for $s=\Omega\left(n\right)$.

Instead of considering the threshold bias $b^{*}$, we shift the focus
to the maximal component size $s$ achievable by \textsc{Maker} in
the $\left(1:b\right)$ component game, for a given bias $b$ (assuming optimal play of both players). Let
us denote this quantity by $s_{b}^{*}\left(G\right)$. Bednarska and
\L uczak considered in~\cite{BL-01} the $\left(1:b\right)$ component
game on the complete graph. They showed that $s_{b}^{*}\left(K_{n}\right)$
undergoes a certain type of  phase transition around $b=n$; specifically, that $s_{n+t}^{*}\left(K_{n}\right)=\left(1-o\left(1\right)\right)n/t$
for $\sqrt{n}\ll t\ll n$ but $s_{n-t}^{*}\left(K_{n}\right)=t+O\left(\sqrt{n}\right)$
for $0\le t\le n/100$.

The component game on $d$-regular graphs for fixed $d\ge3$ was considered
by Hod and Naor in~\cite{HN-14}. They showed
a similar phase transition of $s_{b}^{*}$ around $b=d-2$: for any
$d$-regular graph $G$ on $n$ vertices,
\[
s_{b}^{*}\left(G\right)=\begin{cases}
O\left(1\right), & b\ge d-1;\\
O\left(\log n\right) & b=d-2
\end{cases}
\]
whereas $s_{d-3}^{*}\left(\mathcal{G}\left(n,d\right)\right)=\Omega\left(n\right)$
WHP.

\subsection{\label{sec:our-results}Our results}

Given previous results, it is not surprising that the component game
on the binomial random graph $\mathcal{G}\left(n,p\right)$
undergoes a phase transition too. Writing $p=c/n$ \textemdash{} so
the expected average degree in $\mathcal{G}\left(n,p\right)$ is $c$
\textemdash{} we could perhaps guess the phase transition occurs around
$c=b+2$ in accord with the results for $d$-regular graphs. Another
plausible approach would be to consider the so-called \emph{random
graph intuition}, as first observed by Chv\'{a}tal and Erd\H{o}s in~\cite{CH-78}:
it turns out that in many cases, the winner of a game in which both
sides play to their best is the same as if both sides were to play
randomly. We may thus guess that the transition occurs around $c=b+1$,
since in the random players scenario \textsc{Maker} would end up with
the edge set of $\mathcal{G}\left(n,p/\left(b+1\right)\right)$, which
contains a linear-size component if and only if $c>b+1$.

It turns out, however, that neither of these heuristics gives the correct answer. The key graph parameter is degeneracy, which is related to the minimum degree of subgraphs,
rather than average degree; consequently, the critical $c$ is (somewhat) larger
than $b+2$.
\begin{defn*}
For an integer $k\ge1$, the $k$\emph{-core} of a graph $G=\left(V,E\right)$
is its largest subgraph $K$ of minimum degree $\delta\left(K\right)\ge k$.
If no such subgraph exists, we say $G$ has an empty $k$-core, or
that $G$ has no $k$-core, or that $G$ is $\left(k-1\right)$-\emph{degenerate}.
\end{defn*}
Pittel, Spencer and Wormald~\cite{PSW-96} (see also~\cite{CW-06,JL-07,Rio-08})
proved that for every $k\ge3$ there exists a threshold constant $c_{k}$
for the appearance of the $k$-core in the binomial
random graph. That is, $\ggnc$ has WHP an empty $k$-core for $c<c_{k}$
and a linear $k$-core for $c>c_{k}$. It was previously shown by
\L uczak~\cite{Luczak-91,Luczak-92} that $\ggnp$ WHP has either
an empty or a linear-size $k$-core, for every fixed $k\ge3$. The constant $c_k$ is implicitly defined and satisfies $c_{k}=k+\sqrt{k\log k}+O \big(\sqrt{k/\log k} \big)$ (see~\cite[Lemma 1]{PVW-11}). For small values of $k$ we have
$c_{3}\approx3.351,$ $c_{4}\approx5.149$, $c_{5}\approx6.799$, $c_{6}\approx8.365$.

\medskip{}
The following theorems show that for any constant $b$, the phase transition for the $(1:b)$ component
game on $\ggnc$ occurs at $c=c_{b+2}$.
\begin{thm}
\label{thm:breaker-win}For any two constants $b$ and $c<c_{b+2}$ we have WHP $s_{b}^{*}\left(\gnc\right)=o\left(\log^{3}n\right)$.
\end{thm}

\begin{thm}
\label{thm:maker-win}For any two constants $b$ and $c>c_{b+2}$ we have WHP $s_{b}^{*}\left(\gnc\right)=\Omega\left(n\right)$.
\end{thm}

Fix an integer $k\ge3$. The standard algorithm for finding the
$k$-core of a graph $G$ on $n$ vertices is the following, called
the $k$-peeling process: starting from $G_{0}=G$, let $\left(G_{t}\right)_{t\ge0}$
be the sequence of subgraphs of $G$, where $G_{t+1}$ is obtained
from $G_{t}$ by deleting all edges incident with vertices of degree
at most $k-1$. Since $G_{t+1}\subseteq G_{t}$, this (deterministic)
process stabilises after some finite time $T^*\left(G\right)\le n$.
If $G$ is $\left(k-1\right)$-degenerate, the graph $G_{ T^*}$
is empty; otherwise, $G_{ T^*}$ is the $k$-core of $G$, plus,
possibly, some isolated vertices (we only delete edges, so $G_{t}$
has $n$ vertices for all $t\ge0$, although some of them may become
isolated along the process).
\begin{rem*}
Two variations of this process are: $\left(i\right)$ delete vertices
instead of edges; $\left(ii\right)$ process in every step only a
single vertex of degree less than $k$.
\end{rem*}
Jiang, Mitzenmacher, and Thaler used a branching process argument
in~\cite{JMT-16} (see also~\cite{Gao-14}) to bound the typical value of the stabilization
time $T^*\left(n,c\right)$ of the peeling process on $\gnc$
in the subcritical regime $c<c_{k}$:
\begin{thm}[{\cite[Theorems~1 and~2]{JMT-16}}]
\label{thm:stabilization-time}Fix $k\ge3$ and $c<c_{k}$. Then WHP
\[T^*\left(n,c\right)=\log_{k-1}\log n+\Theta\left(1\right).\]
\end{thm}

\begin{rem*}
For $c>c_{k}$, the typical stabilization time is $T^*\left(n,c\right)=\Theta\left(\log n\right)$;
see~\cite[Theorem~4]{AM-15} for the upper bound and~\cite[Theorem~3]{JMT-16}
for the lower bound.
\end{rem*}
In Section~\ref{sec:technical-background} we analyse  the
peeling process further, and prove a related result, which may be of independent
interest. While it takes $\log_{k-1}\log n+\Theta\left(1\right)$ time by
Theorem~\ref{thm:stabilization-time} to peel the \emph{entire} graph,
the graph is WHP already shattered into tiny fragments after a \emph{constant
}number of iterations:
\begin{thm}
\label{thm:shatter-time}Fix $k\ge3$ and $c<c_{k}$. There exists
a constant $t^{\dagger}=t^{\dagger}\left(c\right)$ such that, in
the $k$-peeling process on $\ggnc$, WHP all connected components
of $G_{t^{\dagger}}$ have size $o\left(\log^{3}n\right)$.
\end{thm}

\begin{rem*}
Our proof actually yields the somewhat better bound, $O\left(\log^3n/\log^2\log n\right)$, in Theorem~\ref{thm:shatter-time}, and consequently in Theorem~\ref{thm:breaker-win}. We use the term $o\left(\log^3n\right)$ in both theorems for brevity.
\end{rem*}

In Section~\ref{sec:prelim}
we provide some notations and a few bounds we use throughout the paper.
In Section~\ref{sec:breaker}, we prove Theorem~\ref{thm:breaker-win},
by showing how \textsc{Breaker} can WHP limit the radius of \textsc{Maker}'s
components (via Theorem~\ref{thm:shatter-time}) and hence limit
their size to poly-logarithmic. In Section~\ref{sec:maker} we prove
Theorem~\ref{thm:maker-win}, by showing how \textsc{Maker} can WHP
build a tree of linear size within the $\left(b+2\right)$-core of
$\gnc$. This is done in a two-step strategy:
\textsc{Maker} carefully nurtures a small sapling, which she afterwards
grows arbitrarily into a full-scale tree. The proofs of Theorem~\ref{thm:shatter-time}
and of two ingredients of the proof of Theorem~\ref{thm:maker-win}
are rather technical; we provide required background in Section~\ref{sec:technical-background}
and the proofs themselves in Section~\ref{sec:technical-proofs}.

\section{\label{sec:prelim}Preliminaries}

\subsection{Notation}

We use standard graph theory terminology, and in particular the following.
For a given graph $G$ we denote by $V(G)$ and $E(G)$ the set of
its vertices and the set of its edges, respectively. The \emph{excess}
of $G$ is defined as $\xs\left(G\right)=\left|E\left(G\right)\right|-\left|V\left(G\right)\right|$.
For two disjoint sets of vertices $A,B\subseteq V$ we denote by $E(A,B)$
the set of all edges $ab\in E$ with $a\in A$ and $b\in B$. For
a subgraph $F\subset G$ we denote its \emph{edge boundary} (with
respect to $G$) by $\partial F=E\left(V\left(F\right),V\left(G\right)\setminus V\left(F\right)\right)$.
For a positive integer $t$, the \emph{$t$-neighbourhood} of a vertex
$v\in V\left(G\right)$, also referred to as the \emph{ball} of radius
$t$ around $v$, is the subset of vertices of $G$ whose distance
to $v$ is at most $t$.

\medskip{}
Given a rooted tree $T$ and an integer $k\ge3$, a non-leaf vertex
$v\in V\left(T\right)$ is called \emph{$k$-light} if $v$ has less
than $k$ children in $T$ and \emph{$k$-heavy} otherwise. A \emph{tree
path} of $T$ is a sequence $\left(v_{0},v_{1},\ldots,v_{j}\right)$
of vertices of $T$ such that $v_{i-1}$ is the parent of $v_{i}$
for all $i=1,\ldots,j$. A tree path is called \emph{$k$-light} if
all its vertices are $k$-light.
The \emph{level} of any vertex $v\in V(T)$ is the length of the unique path in $T$ between $v$ and the root, and the \emph{height} of $T$ (assuming $T$ is finite) is the maximum level of a vertex $v \in V(T)$.

\medskip{}
For two integers $n>0$ and $m\ge0$, let $\D_{n,2m}$ denote the set of all nonnegative integer vectors $\dd = (d_1,\dots,d_n)$ such that $\sum d_i = 2m$. Each $\dd \in \D_{n,2m}$ is called a \emph{degree sequence}, and denote its maximum degree by~$\Delta(\dd) := \max_i d_i$.

\medskip{}
At any point during the game, an unclaimed edge is called \emph{free}. The
act of claiming one free edge by one of the players is called a \emph{step}.
\textsc{Breaker}'s $b$ successive steps (and \textsc{Maker}'s single
step) are called a \emph{move}. A \emph{round} in the game consists
of one move of the first player, followed by one move of the second
player. Since being the first player is never a disadvantage, we will
assume \textsc{Maker} starts when proving Theorem~\ref{thm:breaker-win}, and assume \textsc{Breaker} starts when proving Theorem~\ref{thm:maker-win}.

\medskip{}

Throughout the paper we use the well-known bound $\binom{n}{k}\le n^{k}/k!\le\left(en/k\right)^{k}$
for nonnegative integers~$n$ and~$k$. Let $\left[n\right]_{k}=n\left(n-1\right)\cdots\left(n-k+1\right)$,
and note that $\left[n\right]_{k}=n!/\left(n-k\right)!$ for $n\ge k$.
For a positive integer $m$, let $\left(2m-1\right)!!=\left(2m-1\right)\left(2m-3\right)\cdots3\cdot1$
and $\left(2m\right)!!=\left(2m\right)\left(2m-2\right)\cdots4\cdot2=2^{m}\cdot m!$.

\medskip{}

In this paper we make extensive use of the Poisson distribution. For
an integer $j\ge0$ and a real number $\lambda\ge0$, let $\Psi_{j}\left(\lambda\right)=\prob\left[\Poisson\left(\lambda\right)=j\right]=e^{-\lambda}\lambda^{j}/j!$, let $\Psi_{\ge j}\left(\lambda\right)=\prob\left[\Poisson\left(\lambda\right)\ge j\right]=\sum_{i\ge j}\Psi_{i}\left(\lambda\right)$, and let $\Psi_{<j}\left(\lambda\right)=\prob\left[\Poisson\left(\lambda\right) < j\right]= 1-\Psi_{\ge j}$.
Note that for any real number $\lambda \ge 0$ and for any two integers $j\ge\ell\ge 0$ we have
\begin{equation}\label{eq:poisson-obs}
\left[j\right]_{\ell}\Psi_{j}\left(\lambda\right)
= e^{-\lambda}\lambda^{j}/(j-\ell)!
= \lambda^{\ell}\Psi_{j-\ell}\left(\lambda\right).
\end{equation}

Given a positive integer $\ell>0$ and a real number $\lambda>0$,
let $Z_{\ell}\left(\lambda\right)$ denote an $\ell$-truncated Poisson
random variable, which is a $\Poisson\left(\lambda\right)$ random
variable conditioned on being at least $\ell$. In other words, $\Pr\left[Z_{\ell}\left(\lambda\right)<\ell\right]=0$
and $\Pr\left[Z_{\ell}\left(\lambda\right)=j\right]=\Psi_{j}\left(\lambda\right)/\Psi_{\ge\ell}\left(\lambda\right)$
for $j\ge\ell$.

\medskip{}
Our proofs are asymptotic in nature and whenever necessary, we assume $n$ is large enough. We omit floor and ceiling signs when these are not crucial. All logarithms in this paper, unless specified otherwise, are natural.

\subsection{Local structure of $\protect\gnc$ }

We also provide several results about the typical local structure
of $\gnc$, which will be useful later. We begin with the following
bound on the volume of balls in $\gnc$.
\begin{lem}[{\cite[Lemma 1]{CL-01}}]
\label{lem:volume-of-balls-in-Gnp}Let $\ggnc$ for $c>1$. Then
WHP, for every vertex $v\in V\left(G\right)$ and for every $1\le t\le n$,
there are at most $2t^{3}c^{t}\log n$ vertices in $G$ within distance
$t$ of $v$.
\end{lem}

Next we show that even though $G$ is not acyclic WHP for $c>1$,
we do not expect short cycles; in other words, local neighbourhoods
in $G$ are trees.
\begin{claim}
\label{clm:no-cycles}The probability of having a cycle in the $t$-neighbourhood
of a given vertex $v$ in $\ggnc$ is $o\left(1\right)$ for $t = \lceil \alpha\log n\rceil$, where $\alpha = \alpha(c)$ is some positive constant.
\end{claim}

\begin{proof}
It suffices to prove this for $c>1$. By Lemma~\ref{lem:volume-of-balls-in-Gnp},
a breadth-first search from $v$ discovers WHP at most $s=\left\lfloor 2t^{3}c^{t}\log n\right\rfloor $
vertices in the $t$-neighbourhood of $v$; the probability of having
either a back-edge or a side edge closing a cycle, in addition to the tree edges, among the first $s$ vertices discovered, is bounded
by
\[
\binom{s}{2}c/n<s^{2}c/n\le4t^{6}c^{1+2t}\log^{2}n/n,
\]
which is $o\left(1\right)$ for our choice of $t$ and for sufficiently small $\alpha$ (which only depends on $c$).
\end{proof}
Last, we show that sufficiently large connected subgraphs of $G$
expand by only a constant factor WHP.
\begin{claim}
\label{clm:vertex-expansion}Let $\ggnc$ for $c>1$. Then WHP, every
connected subset $U\subseteq V\left(G\right)$ of size $\left|U\right|\ge\log n$
has at most $3c\left|U\right|$ neighbours in $G$.
\end{claim}

\begin{proof}
Fix a set $U\subseteq V\left(G\right)$ of size $u=\left|U\right|$.
It is connected (and thus contains at least one of the $u^{u-2}$ possible
spanning trees of $U$) with probability at most $u^{u-2}\left(c/n\right)^{u-1}$.
The size of the external neighbourhood of $U$ is distributed $\Bin\left(n-u,1-\left(1-c/n\right)^{u}\right)$, stochastically dominated by
$\Bin\left(n,cu/n\right)$,
and thus by the Chernoff bound, it has more than $\left(1+\delta\right)cu$
neighbours with probability at most $\exp\left(\delta^{3}-\delta^{2}cu/2\right)$.
Therefore, the probability that, for some $u\ge\log n$, there exists
a connected set $U$ of $u$ vertices with more than $3cu$ neighbours
is bounded by
\begin{align*}
\sum_{u=\left\lceil \log n\right\rceil }^{n}\binom{n}{u}u^{u-2}\left(c/n\right)^{u-1}\exp\left(8-2cu\right) & \le\frac{e^{8}n}{c}\sum_{u=\left\lceil \log n\right\rceil }^{n}u^{-2}\left(ce^{1-2c}\right)^{u}\\
 & <\frac{e^{8}n}{c\log^{2}n}\sum_{u=\left\lceil \log n\right\rceil }^{\infty}e^{-u}\le\frac{e^{9}}{c\left(e-1\right)\log^{2}n}=o\left(1\right).\qedhere
\end{align*}
\end{proof}

\section{\label{sec:breaker}His Side}

We now present a strategy for \textsc{Breaker} to use in the $\left(1:b\right)$
component game on $\gnc$ for $c<c_{b+2}$,
and show that it WHP restricts \textsc{Maker}'s connected components
to poly-logarithmic size.

To facilitate the description of \textsc{Breaker}'s strategy, we introduce
some terminology. Denote the \emph{rank} of a vertex~$v$ with respect
to the $\left(b+2\right)$-peeling process $\left(G_{t}\right)_{t\ge0}$
of a graph $G=G_{0}$ by $\rho\left(v\right)=\inf\left\{ t\ge0\mid\deg_{G_{t}}\left(v\right)<b+2\right\} $.
When $G$ is $\left(b+1\right)$-degenerate, $\rho\left(v\right)\le  T^*\left(G\right)$
is finite for all $v$. In particular, this holds WHP for $\ggnc$
with $c<c_{b+2}$.

An edge $uv\in E\left(G\right)$ is called \emph{horizontal} if $\rho\left(u\right)=\rho\left(v\right)$
and \emph{vertical} otherwise. A \emph{horizontal component} (H-comp,
for short)~$C$ is a connected component in the graph consisting
of horizontal edges claimed by \textsc{Maker}, and its \emph{rank}
$\rho\left(C\right)$ is the rank shared by all its vertices. A vertical
edge $e$ is \emph{above} $C$ if $V\left(C\right)$ contains $e$'s
endpoint of lesser rank. Finally, denote by $F\left(C\right)$ the
set of free edges incident with $V\left(C\right)$ in $G_{t}$, where
$t=\rho\left(C\right)$. We partition $F\left(C\right)$ to vertical
and horizontal edges, writing $F\left(C\right)=F_{V}\left(C\right)\cup F_{H}\left(C\right)$.

\paragraph{\textsc{Breaker}'s Strategy $\mathcal{S}_{B}$:}

We assume \textsc{Maker} is the first player, so each round consists
of a move by \textsc{Maker} and a response by \textsc{Breaker}. For
a particular round, let $e=uv$ be the edge claimed by \textsc{Maker}
on her move. Assume without loss of generality that $\rho\left(u\right)\le\rho\left(v\right)$,
and let $C$ be the H-comp containing $u$ in \textsc{Maker}'s graph
after this move. \textsc{Breaker}'s strategy $\mathcal{S}_{B}$ is
to claim in each of his $b$ steps during this round:
\begin{itemize}
\item an arbitrary edge from $F_{V}\left(C\right)$, if $F_{V}\left(C\right)$
is nonempty; otherwise
\item an arbitrary edge from $F_{H}\left(C\right)$, if $F_{H}\left(C\right)$
is nonempty; otherwise
\item an arbitrary free edge.
\end{itemize}
It is clear that \textsc{Breaker} can follow $\mathcal{S}_{B}$ throughout
the game. The following claim characterises the horizontal components
allowed by $\mathcal{S}_{B}$:
\begin{claim}
\label{clm:H-comp-invariants}At the beginning of every round, every
H-comp $C$ satisfies exactly one of the following:
\begin{enumerate}
\item[$(i)$] \textsc{ Maker} claimed exactly one edge above $C$ and $F\left(C\right)$
is empty.
\item[$\left(ii\right)$] \textsc{ Maker} claimed no edges above $C$ and $\left|F\left(C\right)\right|\le\max\left\{ 0,b+2-\left|V\left(C\right)\right|\right\} $.
\end{enumerate}
\end{claim}

\begin{proof}
We prove this by induction on the number of rounds. Before the game
starts, $\left(ii\right)$ holds for every H-comp $C$ since no edges
have been claimed, $\left|V\left(C\right)\right|=1$ and $\left|F\left(C\right)\right|\le b+1$
by definition of $\left(b+2\right)$-rank. Assume the claim holds
at the beginning of round $r\ge0$; we show it holds at the end
of that round. Let $e=uv$ be the edge claimed by \textsc{Maker} on
her $r$th move, and assume WLOG that $\rho\left(u\right)\le\rho\left(v\right)$.
Denote by $C_{u}$ and $C_{v}$, respectively, the H-comps containing
$u$ and $v$ at the beginning of round $r$, and by $C$ the H-comp
containing $u$ after her move. It remains to show that after \textsc{Breaker}'s
move $C$ satisfies either $\left(i\right)$ or $\left(ii\right)$
since no other H-comp is affected by \textsc{Maker}'s move, and \textsc{Breaker}'s
move cannot disrupt an H-comp which was already satisfying $\left(i\right)$
or $\left(ii\right)$ from doing so. We distinguish between two cases:
\begin{casenv}
\item If $e$ is vertical, observe that $C=C_{u}$. Since we had $e\in F\left(C\right)$
before \textsc{Maker}'s move, $e$ must be the first edge \textsc{Maker}
claimed above $C$ by assumption. Before \textsc{Breaker}'s move we
have $\left|F\left(C\right)\right|\le b$ (as $e$ is already claimed),
so he claims all of $F\left(C\right)$ according to $\mathcal{S}_{B}$.
\item If $e$ is horizontal, $C=C_{u}\cup C_{v}$. Since both $F\left(C_{u}\right)$
and $F\left(C_{v}\right)$ contained $e$ at the beginning of the
round, both satisfied $\left(ii\right)$ and thus\textsc{ Maker} claimed
no edges above $C$. Moreover, before \textsc{Maker}'s move we had
$0<\left|F\left(C_{u}\right)\right|\le b+2-\left|V\left(C_{u}\right)\right|$,
and similarly for $v$, thus before \textsc{Breaker}'s move we have
\[
\left|F\left(C\right)\right|=\left|\left(F\left(C_{u}\right)\setminus\left\{ e\right\} \right)\cup\left(F\left(C_{v}\right)\setminus\left\{ e\right\} \right)\right|\le2\left(b+1\right)-\left|V\left(C\right)\right|
\]
 and after his move, according to $\mathcal{S}_{B}$, either $F\left(C\right)$
is empty or $\left|F\left(C\right)\right|\le b+2-\left|V\left(C\right)\right|$.\qedhere
\end{casenv}
\end{proof}
Claim~\ref{clm:H-comp-invariants} yields the following two corollaries:
\begin{cor}
\label{cor:H-comp-size}Throughout the game $\left|V\left(C\right)\right|\le2\left(b+1\right)$
for every H-comp $C$.
\end{cor}

\begin{proof}
Indeed, no free edges are incident with H-comps of size $b+2$ or more
by Claim \ref{clm:H-comp-invariants}, and thus the largest H-comp
possibly achievable by \textsc{Maker} is obtained by claiming a free
horizontal edge between two H-comps of size $b+1$ each.
\end{proof}
\begin{cor}
\label{cor:T-gamma}Given a connected component $\Gamma$ of \textsc{Maker}'s
graph, if we contract every H-comp in $\Gamma$ to a single vertex,
we get a tree $T_{\Gamma}$ of height at most $\rho\left(C_{0}\right)$,
where $C_{0}$ is the H-comp in $\Gamma$ of maximal rank (that is,
the H-comp corresponding to the root of $T_{\Gamma}$).
\end{cor}

At this point we can already establish a weaker version of Theorem~\ref{thm:breaker-win}.
Indeed, using the above corollaries, Lemma~\ref{lem:volume-of-balls-in-Gnp}
and Theorem~\ref{thm:stabilization-time}, we obtain the following
poly-logarithmic bound on the size of \textsc{Maker}'s connected components.
\begin{prop}
\label{prop:breaker-strategy-is-ok}For $c<c_{b+2}$, $\ggnc$ is
WHP such that by playing according to $\mathcal{S}_{B}$, \textsc{Breaker}
ensures that all connected components in \textsc{Maker}'s graph are
of size $O\left(\log^{4b+5}n\right)$.
\end{prop}

\begin{proof}
Fix a connected component $\Gamma$ in \textsc{Maker}'s graph. The
height of $T_{\Gamma}$ from Corollary~\ref{cor:T-gamma} is bounded
by the stabilization time, and we have WHP $T^*= T^*\left(n,c\right)\le\log_{b+1}\log n+O\left(1\right)$
by Theorem~\ref{thm:stabilization-time}. Moreover, the distance
between two vertices in the same H-comp is at most $2b+1$ by Corollary~\ref{cor:H-comp-size},
so $\Gamma$ is contained in a ball of radius $t=2\left(b+1\right) T^*$
around any vertex $v\in V\left(\Gamma\right)$ of maximal rank. Applying
Lemma~\ref{lem:volume-of-balls-in-Gnp} yields the bound
\[
2t^{3}c^{t}\log n=O\left(\log^{3}\log n\cdot c^{2\left(b+1\right)\log_{b+1}\log n}\cdot\log n\right)=O\left(\log^{1+2\left(b+1\right)\log_{b+1}c}n\cdot\log^{3}\log n\right).
\]
Note that $c<c_{b+2}<2\left(b+1\right)$ and thus
\[
1+2\left(b+1\right)\log_{b+1}c<1+2\left(b+1\right)\left(1+\log_{b+1}2\right)\le4b+5,
\]
establishing the bound on the size of $C$.
\end{proof}
Proposition~\ref{prop:breaker-strategy-is-ok} gives a nonuniform
poly-logarithmic bound (i.e., the exponent depends on $b$, even when
computed precisely). Using Theorem~\ref{thm:shatter-time} and Claim~\ref{clm:vertex-expansion},
we improve the bound to a uniform $o\left(\log^{3}n\right)$ for any
value of $b$, and establish Theorem~\ref{thm:breaker-win}, via
the following proposition:
\begin{prop}
\label{prop:breaker-strategy-is-awesome}For $c<c_{b+2}$, $\ggnc$
is WHP such that if \textsc{Breaker} plays according to $\mathcal{S}_{B}$,
all connected components in \textsc{Maker}'s graph are of size $o\left(\log^{3}n\right)$.
\end{prop}

\begin{proof}
Fix a connected component $\Gamma$ in \textsc{Maker}'s graph. If
$\left|V\left(\Gamma\right)\right|<\log n$, we are done; otherwise,
let $U'\subseteq V\left(\Gamma\right)$ be the set of vertices of
rank at least $t^{\dagger}$, where $t^{\dagger}=t^{\dagger}\left(c\right)$
is the constant from Theorem~\ref{thm:shatter-time}, and let $U$
be the vertex set of a connected $\Gamma'\subseteq\Gamma$ of minimal
size such that $U'\subseteq U$ and $\left|U\right|\ge\log n$. By
Theorem~\ref{thm:shatter-time}, $\left|U\right|=o\left(\log^{3}n\right)$.
Furthermore, $\Gamma$ is contained in a ball of radius $t=2\left(b+1\right)t^{\dagger}$
around $U$, thus $t$ applications of Claim~\ref{clm:vertex-expansion}
yield the desired bound
\[
\left|V\left(\Gamma\right)\right|\le\left(3c\right)^{t}\left|U\right|=O\left(\left|U\right|\right)=o\left(\log^{3}n\right).\qedhere
\]
\end{proof}
\begin{rem*}
A slightly refined strategy by \textsc{Breaker} would be to claim
edges from $F_{H}\left(C\right)$ before edges from $F_{V}\left(C\right)$
if $\left|F_{V}\left(C\right)\right|>b$. It is not very hard to see
that by this he makes sure that throughout the game all H-comps except
the ``root'' ones are of size at most two. This improves the exponent
in Proposition~\ref{prop:breaker-strategy-is-ok} to $3+\log_{b+1}4$,
which approaches~$3$ from above as~$b$ increases. However, for
all $b\ge1$ this is still inferior to the bound given in Proposition~\ref{prop:breaker-strategy-is-awesome}
(which does not benefit from the modified strategy). We thus chose
to present the simpler strategy.
\end{rem*}

\section{\label{sec:maker}Her Side}

It is quite easy to prove a weaker version of Theorem~\ref{thm:maker-win}:
that WHP $s_{b}^{*}\left(\gnc\right)=\Omega\left(n\right)$
when $c>c_{b+3}$. Indeed, in this case $\mathcal{K}'$, the $\left(b+3\right)$-core
of $\gnc$, is WHP of linear size and a very
good edge expander. It follows that the naive tree-building strategy
within $\mathcal{K}'$ (i.e., the strategy that builds a single tree
$T\subset\mathcal{K}'$ by repeatedly claiming an arbitrary free edge
of $\partial T$, starting from an arbitrary vertex of $\mathcal{K}'$)
is successful for \textsc{Maker}, since as long as $T$ is sublinear in size, i.e., has $o(n)$ vertices, we have
$\left|\partial T\right|>b\left|V\left(T\right)\right|$ and \textsc{Breaker}
cannot claim all the boundary of $T$. This is the same strategy \textsc{Maker}
uses in~\cite{HN-14} when playing on a random $\left(b+3\right)$-regular
graph.

For $c_{b+2}<c<c_{b+3}$, however, the above strategy applied to~$\mathcal{K}$, the $\left(b+2\right)$-core of $\ggnc$, is prone to
being shut down by \textsc{Breaker} in the early stages of the game
despite the minimum degree in $\mathcal{K}$ being $b+2$. Thus in the proof of Theorem~\ref{thm:maker-win},
\textsc{Maker} uses a tree-building strategy that begins in a more
refined manner. In the tree $T\subset\mathcal{K}$ \textsc{Maker}
builds, she makes sure to have many heavy vertices (i.e., of degree
at least $b+3$ in $\K$). Only after securing sufficiently many such
vertices does she continue by growing $T$ naively.

\medskip{}

 Recall that for a rooted tree, a non-leaf vertex is $k$-light if it has less
than $k$ children, and it is $k$-heavy otherwise, and that a tree path (a descending path towards the leaves) is $k$-light if it consists solely of $k$-light vertices.
Note that since leaves are not considered light, any tree path terminating in a leaf is not $k$-light by definition (although permitting leaves to be light would require only a trivial modification of our arguments).
\medskip

For simplicity, in the remainder of the paper  $k$  stands for $b+2$, and $\K$ is the $k$-core of $G$.
 The proof of Theorem~\ref{thm:maker-win} contains three main ingredients, stated in three lemmas.  The first two lemmas require the following definition.
\begin{defn*}
An $\left(N,L\right)$-tree is a balanced tree (i.e., all of its leaves are at the same level) of height $N$ with no $\left(k-1\right)$-light vertices and no $k$-light tree paths of length $L$.
\end{defn*}

\begin{lem}
\label{lem:accumulate-heavy}In the $\left(1:k-2\right)$ \textsc{Maker\textendash Breaker}
game played on the edge set of an $\left(N,L\right)$-tree $T^*$,
\textsc{Maker} has a strategy to build a subtree $T\subset T^*$
with at least $\alpha^{N}$ vertices that are $k$-heavy
in $T^*$, for some constant $\alpha=\alpha\left(k,L\right)>1$.
Moreover, using this strategy \textsc{Maker}'s graph is a tree throughout
the game, until $T$ is achieved.
\end{lem}

We assume the hypothesis of Theorem~\ref{thm:maker-win}, i.e., $c>c_k$, in the following two lemmas.

\begin{lem}
\label{lem:good-tree}  $\mathcal{K}$ contains WHP an $\left(N,L\right)$-tree
for $N=\log^{2}\log n$ and some constant $L=L\left(k,c\right)$.
\end{lem}

\begin{lem}
\label{lem:no-bad-two-cores}There exists some $\eps=\eps\left(k,c\right)>0$
such that WHP any subgraph $C\subset\mathcal{K}$ with minimum degree
at least~$2$, $\xs\left(C\right)\ge\log^{4}n$ and $\left|\partial C\right|\le (k-2)\left|V\left(C\right)\right|$
satisfies $\left|V\left(C\right)\right|\ge\eps n$.
\end{lem}

Lemma~\ref{lem:accumulate-heavy} is proved in Section~\ref{sec:accumulate-heavy}, while the proofs of Lemmas~\ref{lem:good-tree} and~\ref{lem:no-bad-two-cores}
are postponed to Sections~\ref{sec:good-tree} and~\ref{sec:no-bad-two-cores}, respectively.

Next we show how these lemmas imply Theorem~\ref{thm:maker-win}, but first we need the following easy claim.
\begin{claim}
\label{clm:2-core}Let $F$ be an induced subgraph of $\mathcal{K}$
with non-negative excess such that $\left|\partial F\right|\le b\left|V\left(F\right)\right|$,
and let $C$ be the $2$-core of $F$ (note that $C$ is not empty since $\xs\left(F\right)\ge0$ and thus
$F$ contains a cycle). Then $\xs\left(C\right)\ge\xs\left(F\right)$ and $\left|\partial C\right|\le b\left|V\left(C\right)\right|$.
\end{claim}

\begin{proof}
In order to obtain $C$ from $F$ we use a sequential vertex-deleting
variant of the $2$-peeling process. Starting from $F_{0}=F$, as
long as the minimum degree of $F_{t}$ is less than two, we obtain
$F_{t+1}$ from $F_{t}$ by deleting one arbitrary vertex $v\in V\left(F_{t}\right)$
of degree $\deg_{F_{t}}\left(v\right)\le1$. Regardless of the order
of deletions, this process terminates with the 2-core $C=F_{ T^*}$.
We prove that $\xs\left(F_{t}\right)\ge\xs\left(F\right)$ and $\left|\partial F_{t}\right|\le b\left|V\left(F_{t}\right)\right|$
for all $t\ge0$ by induction.

For $t=0$ there is nothing to prove; assuming it holds for $t$,
let $v\in V\left(F_{t}\right)$ be the vertex selected for deletion,
whose degree in $F_{t}$ is $0$ or $1$. In $F_{t+1}$ there is one
less vertex (i.e., $v$) and at most one less edge than $F_{t}$,
hence $\xs\left(F_{t+1}\right)\ge\xs\left(F_{t}\right)\ge\xs\left(F\right)$;
furthermore
\[
\partial F_{t+1}=\partial F_{t}\cup E\left(v,V\left(F_{t+1}\right)\right)\setminus E\left(v,V\left(\mathcal{K}\setminus F_{t}\right)\right)
\]
so
\[
\left|\partial F_{t+1}\right|=\left|\partial F_{t}\right|+\deg_{F_{t}}\left(v\right)-\left(\deg_{\mathcal{K}}\left(v\right)-\deg_{F_{t}}\left(v\right)\right)\le b\left|V\left(F_{t}\right)\right|+1-\left(b+1\right)=b\left|V\left(F_{t+1}\right)\right|.\qedhere
\]
\end{proof}
\begin{proof}[Proof of Theorem~\ref{thm:maker-win}]
By Lemma~\ref{lem:good-tree} \textsc{Maker} can WHP locate a $\left(\log^{2}\log n,L\right)$-tree
$T^*$ in $\mathcal{K}$. She first builds a subtree $T\subset T^*$
with at least $\alpha^{\log^{2}\log n}\gg2\log^{4}n$ vertices of
degree at least $b+3$ in $\K$, which she can do by Lemma~\ref{lem:accumulate-heavy}. She then proceeds naively by claiming in every move an arbitrary
free edge in $\partial T$ as long as possible. \textsc{Maker} can
no longer proceed with her strategy only when \textsc{Breaker} has
claimed the entire boundary of $T$, and in particular $\left|\partial T\right|\le b\left|V\left(T\right)\right|$.
When  this happens, consider the subgraph $F\subset\mathcal{K}$ induced by $V\left(T\right)$.
Clearly $V\left(F\right)=V\left(T\right)$ and thus $\partial F=\partial T$.
Now $F$ satisfies
\begin{align*}
  2\left|E\left(F\right)\right|  =\sum_{v\in V\left(F\right)}\deg_{F}\left(v\right) & =\sum_{v\in V\left(F\right)}\deg_{\mathcal{K}}\left(v\right)-\left|\partial F\right| \\
  & \ge \left(b+2\right)\left|V\left(T\right)\right|+2\log^{4}n-b\left|V\left(T\right)\right| = 2\left(\left|V\left(F\right)\right|+\log^{4}n\right),
\end{align*}
meaning $\xs\left(F\right)\ge\log^{4}n$. By Claim~\ref{clm:2-core}
the 2-core $C$ of $F$ has large excess and small boundary, namely
$\xs\left(C\right)\ge\xs\left(F\right)\ge\log^{4}n$ and $\left|\partial C\right|\le b\left|V\left(C\right)\right|$.
Finally, by Lemma~\ref{lem:no-bad-two-cores}, WHP $C$ has at least
$\eps n$ vertices and thus
\[
\left|V\left(T\right)\right|=\left|V\left(F\right)\right|\ge\left|V\left(C\right)\right|\ge\eps n,
\]
establishing the theorem.
\end{proof}

\subsection{\label{sec:accumulate-heavy}Proof of Lemma~\ref{lem:accumulate-heavy}:
accumulating heavy vertices}

 An $\left(N,L\right)$-tree is \emph{simple} if its root is $k$-light, and all $k$-heavy vertices in the tree have exactly $k$ children. Clearly, any $\left(N,L\right)$-tree contains a simple $\left(N,L\right)$-tree as a subtree, so we may assume without loss of generality that $T^*$ is simple. Note that $w$, the root of $T^*$, has degree $k-1$, every other $k$-light vertex but the leaves has degree $k$, and all $k$-heavy vertices of $T^*$, to which we refer from now on simply as heavy, have degree $k+1$. In this subsection $l(v)$ denotes the level of any vertex $v\in V( T^*)$; in particular, $l(w)=0$ and $l(v)=N$ for every leaf $v$. Whenever we refer to an edge $uv\in E( T^*)$ we assume $l(u)<l(v)$, and define the level of~$uv$ to be~$l(v)$.

We now describe \textsc{Maker}'s strategy $\mathcal{S}_{\textsc{M}}$. Throughout the game, \textsc{Maker}'s graph is a single tree $T=T^{(r)}\subset T^*$, where $T^{(r)}$ denotes her tree after the $r$th round. Initially, $T^{\left(0\right)}$ consists of $w$ solely. In
each move, as long as there exist free edges of level smaller than $N$ in $\partial T$, \textsc{Maker} enlarges her tree by claiming one of these edges, selecting an arbitrary edge of minimum level. The game stops when \textsc{Maker} cannot proceed with this strategy, i.e., when all edges in $\partial T$ have already been claimed by \textsc{Breaker} (note that by definition none of them could have been previously claimed by \textsc{Maker}), except perhaps for some edges in level $N$.

Given $\mathcal{S}_{\textsc{M}}$, we may assume that \textsc{Breaker} only claims edges from $\partial T$ as well. Indeed, suppose that \textsc{Breaker}, according to his strategy, wishes to claim an edge $uv \not\in \partial T$ in one of his steps. He can claim instead the (unique) edge $u'v' \in  \partial T$ such that $v'$ is an ancestor of $v$. By $\mathcal{S}_{\textsc{M}}$, this prevents \textsc{Maker} from claiming any edge in the subtree of $T^*$ rooted at $v'$, and particularly the edge $uv$. If the edge $u'v'$ was already claimed by \textsc{Breaker}, then he claims an arbitrary free edge from $\partial T$ (if none of those exists the game ends anyway). It is evident that \textsc{Breaker} cannot be harmed from this modification of his strategy, and our assumption is justified.

We now assume that \textsc{Maker} plays second and follows $\mathcal{S}_{\textsc{M}}$ (it is trivial to see that she can do so), and show that it is a winning strategy for her, that is, we show that when the game is stopped, her tree $T$ contains sufficiently many heavy vertices. Before doing so we need some additional terminology. First, since both players only claim edges from $\partial T$ by assumption, we naturally redefine free edges to be unclaimed edges from $\partial T$ (instead of all unclaimed edges in $T^*$) for the remainder of this subsection. For $j=1,2,\ldots,N-1$, level $j$ is \emph{complete}
when no free edges remain in level $j$ or less. An edge $uv\in E(T)$ \emph{survives} in level $j$ for $j\ge l\left(v\right)$ if $T$ contains a level $j$ descendant of $v$. Finally, we define $C=C(k,L)=1+k^{L+1}$.

\begin{claim}
\label{clm:FreeEdgesNum} \emph{For $r\ge0$ let $h_{r}$ denote the
number of heavy vertices in $T^{(r)}$, and for $r\ge1$
let $f_{r}$ denote the number of free edges before }\textsc{Maker}\emph{'s
$r$th move. Then $f_{r+1} = h_{r}+1$ for every $r\ge0$.}
\end{claim}

\begin{proof}
Since $T^*$ is simple, and since \textsc{Maker} never claims an edge in level $N$, it follows
that
\[
\left|\partial T^{(r)}\right|=\sum_{v\in V\left(T^{(r)}\right)}\deg_{ T^*}(v)-2\left|E\left(T^{(r)}\right)\right|=(k-1)+rk+h_{r}-2r.
\]
In each of his first $r+1$ moves \textsc{Breaker} claims $b=k-2$
edges, each of them incident with some $v \in V(T)$, so $f_{r+1} = \left|\partial T^{(r)}\right|-(r+1)(k-2)=h_{r}+1$.
\end{proof}
An immediate corollary of Claim~\ref{clm:FreeEdgesNum} is that at
the end of the game $T$ is of height $N-1$, and level $N-1$ is complete. Indeed, since there exist free edges before each of \textsc{Maker}'s moves, the game stops only when all free edges are in level $N$. Having a free edge in level $N$ means that $T$ must have reached level $N-1$.
\begin{claim}
\label{cl:constant} Let $s$ and $j<N-L$
be two integers, and assume that there are~$s$ free edges in level~$j$ right before \textsc{Maker} claims her first edge
in level~$j$. Then at least~$s/C$ of them will be claimed by \textsc{Maker} and survive in level~$j+L$.
\end{claim}

\begin{proof}
Consider first the situation right before \textsc{Maker} claims her first edge in level $j$, and denote the $s$ free edges in this level by $u_{1}v_{1},\dots,u_{s}v_{s}$ (the
parents $u_{i}$ are not necessarily distinct, but the children $v_{i}$
are). For each $1\le i \le s$, let $T_{i}\subset T^*$ be the subtree rooted at~$u_{i}$
consisting of the edge $u_{i}v_{i}$ and the subtree of $T^*$
rooted at $v_{i}$ of height $L$. Note that all free edges are in level $j$ at this point by $\mathcal{S}_{\textsc{M}}$. By the strategies of \textsc{Maker} and \textsc{Breaker}, it follows that all edges in each $T_i$ are still unclaimed, and that exactly one of them, namely $u_iv_i$, is considered free in our new terminology.

Now let us examine the game when level $j+L$ is complete. Denote by $M\subseteq\left\{ 1,2,\ldots,s\right\} $
the set of indices of the edges that survived in level $j+L$ of $T$,
and by $B$ its complement, so $|M|+|B|=s$. We need to show that
$|M|>s/C$.

Let $m_{i}$ and $b_{i}$ denote the number of steps that were played
in $T_{i}$ by \textsc{Maker} and \textsc{Breaker}, respectively.
Recalling that $T^*$ is simple, for every $i=1,2,\ldots,s$, we have the trivial bound
\[
m_{i}\le|E(T_{i})|\le\sum_{t=0}^{L}k^{t}=\frac{k^{L+1}-1}{k-1} < \frac{k^{L+1}}{k-2}.
\]
Now let $i\in B$. Since \textsc{Maker} did not reach level $L$ in
$T_{i}$, i.e., did not reach any of its leaves, it follows by the assumption on \textsc{Breaker}'s strategy that every
step \textsc{Maker} made in $T_{i}$ increased the number of free edges in this
tree by at least $k-2$. Since no free edges remain in $T_{i}$, and
there was one free edge there at the beginning of the analysis, it
follows that $b_{i}\geq m_{i}(k-2)+1$.

During this analysis, which begins with \textsc{Maker}'s move, she
only plays in $T_{1},\ldots,T_{s}$. Thus $\sum_{i=1}^{s}b_{i}\leq(k-2)\sum_{i=1}^{s}m_{i}$.
The left hand side can be bounded from below by
\[
\sum_{i=1}^{s}b_{i}\geq\sum_{i\in B}b_{i}\geq\sum_{i\in B}\left(m_{i}(k-2)+1\right)=(k-2)\sum_{i\in B}m_{i}+|B|,
\]
while the right hand side can be bounded from above by
\[
(k-2)\sum_{i=1}^{s}m_{i}=(k-2)\left(\sum_{i\in B}m_{i}+\sum_{i\in M}m_{i}\right)<(k-2)\sum_{i\in B}m_{i}+|M|k{}^{L+1}.
\]
Putting it all together, we get $|B|<|M|k^{L+1}$, which implies $s=|M|+|B|<C|M|$,
thus the proof is complete.
\end{proof}
\begin{claim}
\label{cl:heavy} For every $0\le i\le\lfloor N/(C(L+1))\rfloor$,
when level $iC(L+1)$ is complete, $T$ contains at least $(2^{i}-1)C$
heavy vertices.
\end{claim}

\begin{proof}
We prove by induction on $i$. The claim holds trivially for $i=0$.
Assuming it holds for $i$, we show that it holds for $i+1$ as well.
Let $0\le j<C$ and write $J_{j}^{i}=(iC+j)(L+1)$. By the induction
hypothesis and by Claim~\ref{clm:FreeEdgesNum}, right before \textsc{Maker}
claims her first edge in level $J_{j}^{i}+1>iC\left(L+1\right)$ there
are at least $(2^{i}-1)C+1$ free edges, all of them in level $J_{j}^{i}+1$
by $\mathcal{S}_{M}$. By Claim~\ref{cl:constant}, at least $2^{i}$
of these free edges will be claimed by \textsc{Maker} and survive
in level $J_{j+1}^{i}$, resulting in at least $2^{i}$ vertex disjoint
tree paths of length $L$ in $T$, each of them containing at least
one heavy vertex by the property of $T^*$. It follows that $T$
contains at least $2^{i}C$ heavy vertices between levels $J_{0}^{i}+1$
and $J_{0}^{i+1}$. By the induction hypothesis $T$ also contains
at least $(2^{i}-1)C$ heavy vertices until level $J_{0}^{i}$, and
the claim holds.
\end{proof}
The game ends when level $N-1$ is complete, and~$T$ then contains by Claim~\ref{cl:heavy}
at least $(2^{\lfloor (N-1)/(C(L+1))\rfloor}-1)C>\alpha^{N}$
heavy vertices for an appropriate $\alpha=\alpha\left(k,L\right)>1$,
establishing Lemma~\ref{lem:accumulate-heavy}.

\section{\label{sec:technical-background}Technical Background}

In this section we describe the technical background required for
Section~\ref{sec:technical-proofs}.

\subsection{The configuration model }

We begin this section with a description of the so-called \emph{configuration model}, introduced by Bollob{\'a}s~\cite{B-80}, which is extremely useful for generating random graphs with a given degree sequence.

Given two positive integers $n$ and $m$, fix an arbitrary degree sequence $\dd \in \D_{n,2m}$. Let $V = \{v_1,\dots,v_n\}$ be a set of vertices, where each $v_i$ is incident with $d_i$ labelled half-edges. Let $W$ denote the set of all half-edges, and let $F$ be a partition of $W$ into $m$ pairs; such a partition, which may also be viewed as a perfect matching of the half-edges, is called a \emph{configuration}. Note that there are exactly $(2m-1)!!$ configurations for $\dd$. By forming an edge from any two half-edges which belong to same pair in $F$, we obtain a multigraph $H = H(F)$ on the vertex set $V$, such that $d_H(v_i) = d_i$.

Let $\Omega^{*}(n,\dd) = \{H(F) \mid \textrm{$F$ is a configuration for $\dd$}\}$ be the set of all multigraphs on $n$ labelled vertices with degree sequence $\dd$, and let $\mathcal{G}^{*}(n,\dd)$ be the probability space of $\Omega^{*}(n,\dd)$ when $F$ is chosen uniformly at random from all possible configurations for $\dd$. Let $\Omega(n,\dd)$ be the set of all simple graphs in $\Omega^{*}(n,\dd)$, and let $\mathcal{G}(n,\dd)$ be the uniform distribution over $\Omega(n,\dd)$. We will make use of the following theorem, due to Frieze and Karo{\'n}ski~\cite{FK-15}.

\begin{thm}[{\cite[Theorem~10.3]{FK-15}}]\label{thm:prob-simple}
Let $\dd = (d_1,\dots,d_n)$ and assume that $\Delta(\dd) \le n^{1/6}$ and $\sum_{i=1}^{n} [d_i]_2 = \Omega(n)$. Then for any multigraph property $\P$
$$\Pr_{G \sim \mathcal{G}(n,\dd)}\left[G \in \P\right] \le (1+o(1))e^{\lambda(\lambda+1)}\cdot\Pr_{G^* \sim \mathcal{G}^*(n,\dd)}\left[G^* \in \P\right],$$
where
$\lambda = \lambda(\dd) = \frac12 \sum_{i=1}^{n}[d_i]_2/ \sum_{i=1}^{n}d_i$.
\end{thm}

\subsection{\label{sec:branching-gnp}Exploring $\gnc$
via a Poisson branching process}

The main ingredient in the proof of Theorem~\ref{thm:shatter-time} is coupling local behavior
in $\gnc$ with an appropriate branching
process. We describe the coupling quickly, borrowing much of the notation
from~\cite{Rio-08}.
We also refer the reader to~\cite{CW-06,JL-07}. Recall that $\Psi_{j}\left(\lambda\right)$, $\Psi_{\ge j}\left(\lambda\right)$ and $\Psi_{<j}\left(\lambda\right)$ denote the probabilities of a $\Poisson(\lambda)$ random variable being equal to $j$, at least $j$, or less than $j$, respectively.

Let $X_{c}$ be a Galton\textendash Watson branching process that
starts with a single particle $x_{0}$ in generation zero, where the
number of children of each particle is an independent $\Poisson$$\left(c\right)$
random variable. Define a sequence $B_{0}\supseteq B_{1}\supseteq\cdots$
of events: for an integer $t\ge0$, let $B_{t}=B_{t}\left(c\right)$
be the event that $X_{c}$ contains a complete $\left(k-1\right)$-ary
tree of height $t$ rooted at $x_{0}$, and let $B=\bigcap_{t\ge0}B_{t}=\lim_{t\to\infty}B_{t}$
be the event that $X_{c}$ contains an infinite $\left(k-1\right)$-ary
tree rooted at $x_{0}$. Denote the probability of $B_{t}$ by $\beta_{t}$
and the probability of $B$ by $\beta=\lim_{t\to\infty}\beta_{t}$.
Then $\beta_{0}=1$. Also, each particle in the first generation of
$X_{c}$ has probability $\beta_{t}$ of having property $B_{t}$,
independently, so the number of such particles is distributed $\Poisson\left(c\beta_{t}\right)$.
Thus, $\beta_{t+1}=\Psi_{\ge k-1}\left(c\beta_{t}\right)$.

 Since $x\mapsto\Psi_{\ge k-1}\left(cx\right)$ is a continuous increasing
function, $\beta=\beta\left(c\right)$ is the maximum solution to
the equation $x=\Psi_{\ge k-1}\left(cx\right)$.  Recall $c_{k}$,
defined in Section~\ref{sec:intro} as the threshold for the appearance
of a nonempty $k$-core in $\gnc$; in terms
of the process $X_{c}$, we have
\begin{equation}
c_{k}=\inf\left\{ c:\beta\left(c\right)>0\right\} .\label{eq:ck-definition}
\end{equation}
\medskip{}

Given a graph $G$ with degree sequence  $(d_1,\dots,d_n)$, the \emph{degree histogram} of   $G$ is the sequence   $D=\left(D_{0},D_{1},\ldots\right)$  such that $D_j=|i:d_i=j|$ for all $j$.

Let us describe the likely degree histogram $D^{t}$ of the
graph $G_{t}$ obtained after $t$ steps of the peeling process. For
$t=0$ the binomial degree distribution of $G_{0}\sim\gnc$
is asymptotically $\Poisson\left(c\right)$. In particular, it is easy to see (e.g., via the second moment method) that WHP
\begin{equation}
D_{j}^{0}\left(n\right)=\begin{cases}
\left(1+o\left(1\right)\right)\Psi_{j}\left(c\right)n, & j= o\left(\log n/\log\log n\right);\\
\Theta\left(\Psi_{j}\left(c\right)n\right), & j= \Theta\left(\log n/\log\log n\right);\\
0, & j= \omega\left(\log n/\log\log n\right).
\end{cases}\label{eq:degree-distribution-gnc}
\end{equation}
Now, for integers $j\ge0$ and $t>0$ let
\begin{equation}
\delta_{j}^{t}=\Psi_{j}\left(c\beta_{t}\right)\Psi_{\ge k-j}\left(c\beta_{t-1}-c\beta_{t}\right).\label{eq:branching-process-degrees-after-t-steps}
\end{equation}
\begin{claim}
\label{clm:degree-distribution-after-t-steps}For fixed $t\ge 0$ and $j\ge 0$
we have $\E\left[D_{j}^{t}\left(n\right)\right]=\left(1+o\left(1\right)\right)\delta_{j}^{t}n$. Moreover, WHP
$D_{j}^{t}\left(n\right)= \left(1+o\left(1\right)\right)\delta_{j}^{t}n$.
\end{claim}

\begin{proof}
We show that a vertex $v$ has degree $j$ in $G_{t}$ with probability
$\left(1+o\left(1\right)\right)\delta_{j}^{t}$. By Claim~\ref{clm:no-cycles}
we only need to consider the case where the local neighbourhood of
$v$ is a tree, i.e., behaves like the branching process $X_{c}$.

Consider first the case $j\ge k$, and note that here $\Psi_{\ge k-j}\left(c\beta_{t-1}-c\beta_{t}\right)=1$.
Each neighbour of $v$ survives $t$ steps of the $k$-peeling process
with probability $\beta_{t}$, independently, and thus the number
of such neighbours is distributed $\Poisson\left(c\beta_{t}\right)$.
Thus, the number of neighbours of $v$ in $G_{t}$ is exactly $j$
with probability $\left(1+o\left(1\right)\right)\Psi_{j}\left(c\beta_{t}\right)$.
For $0<j<k$, having $j$ neighbours is not sufficient for $v$ to
actually survive $t$ steps; we need $v$ to have at least $k$ neighbours
in $G_{t-1}$, of which exactly $j$ have survived $t$ steps. Thus,
the event $B_{t-1}\setminus B_{t}$, whose probability is $\beta_{t-1}-\beta_{t}$,
holds for at least $k-j$ neighbours of $v$. Again, the number of
such neighbours is distributed $\Poisson\left(c\beta_{t-1}-c\beta_{t}\right)$
and thus~$v$ has~$j$ neighbours in~$G_{t}$ and at least~$k-j$ neighbours
in $G_{t-1}\setminus G_{t}$ with probability
$$\left(1+o\left(1\right)\right)\Psi_{j}\left(c\beta_{t}\right)\Psi_{\ge k-j}\left(c\beta_{t-1}-c\beta_{t}\right)=\left(1+o\left(1\right)\right)\delta_{j}^{t}.$$

By linearity of expectation, we have $\E\left[D_{j}^{t}\left(n\right)\right]=\left(1+o\left(1\right)\right)\delta_{j}^{t}n$,
proving the first statement. Calculations similar to the above for
a pair of vertices $u,v$ show that $\var\left[D_{j}^{t}\left(n\right)\right]=o(\E\left[D_{j}^{t}\left(n\right)\right]^{2})$.
Since  $\E\left[D_{j}^{t}\left(n\right)\right]\to\infty$, this establishes the sharp concentration of $D_{j}^{t}\left(n\right)$.
\end{proof}
\begin{rem*}
For every vertex of degree $0$ in $G_t$ we have two options. If the vertex had degree at least $k$ in $G_{t-1}$, then the analysis is identical to that of the case $0<j<k$ in Claim~\ref{clm:degree-distribution-after-t-steps}. Otherwise, less than $k$ of its neighbours survived $t-1$ peeling steps. Thus, WHP
\begin{equation}\label{eq:degree-0-after-t-steps}
D_{0}^{t}\left(n\right)=\left(1+o\left(1\right)\right)\left(\delta_{0}^{t}+\Psi_{<k}\left(c\beta_{t-1}\right)\right)n.
\end{equation}
\end{rem*}

\subsection{\label{sec:new}Degree histogram of the $k$-core}

Considering the peeling process of Section~\ref{sec:branching-gnp},   the limit of the formula in Claim~\ref{clm:degree-distribution-after-t-steps} for large $t$ is suggestive of  the asymptotics  of the degree histogram of the $k$-core $\K$ of $\ggnc$, but it is difficult to make this approach rigorous.
We quote~\cite{CW-06} for a convenient source of the following two lemmas.

With $\beta$ defined as in Section~\ref{sec:branching-gnp}, we define $\mu_c=c\beta$. We  let $\nn$ denote the number   of vertices in $\K$, and $\hat m$ the number of edges.
\begin{lem}\label{CWnm}
The number   of vertices in $\K$   WHP satisfies
\begin{equation}\label{eq:number-of-vertices-in-k-core}
\hat{n} = \left(1+o\left(1\right)\right)\Psi_{\ge k}\left(\mu_{c}\right)n,
\end{equation}
and the number of edges in $\K$   WHP satisfies
\begin{equation}\label{eq:number-of-edges-in-k-core}
\hat{m} = \left(1+o\left(1\right)\right)\frac 12\mu_{c}\Psi_{\ge k-1}\left(\mu_{c}\right)n.
\end{equation}
\end{lem}

\begin{lem}\label{CWdeg}
WHP the degree distribution in $\K$ is asymptotically that of $Z_{k}\left(\mu_{c}\right)$, i.e., the $k$-truncated
Poisson with parameter   $\mu_{c}$. Equivalently, the number of vertices of degree $j$ is 0 for $j<k$, whilst for fixed $j\ge k$ it is
\begin{equation}\label{degree-asy}
(1+o(1))\Pr[Z_{k}\left(\mu_{c}\right)=j]\hat n.
\end{equation}
\end{lem}
 We will also find an alternative expression for the average degree useful.  For every $\mu>0$ and $j\ge k$ we have
\begin{align*}
\frac{\Pr\left[Z_{k}\left(\mu\right)=j\right]}{\Pr\left[Z_{k-1}\left(\mu\right)=j-1\right]} & =\frac{\Psi_{j}\left(\mu\right)/\Psi_{\ge k}\left(\mu\right)}{\Psi_{j-1}\left(\mu\right)/\Psi_{\ge k-1}\left(\mu\right)}\\
 & =\frac{e^{-\mu}\mu^{j}/j!}{e^{-\mu}\mu^{j-1}/\left(j-1\right)!}\cdot\frac{\Psi_{\ge k-1}\left(\mu\right)}{\Psi_{\ge k}\left(\mu\right)}=\frac{\mu}{j}\cdot\frac{\Psi_{\ge k-1}\left(\mu\right)}{\Psi_{\ge k}\left(\mu\right)},
\end{align*}
so, denoting the average degree in $\K$ by $\hat{d}:=2\hat{m}/\hat{n}$, for every $j \ge k$ we have
\begin{equation}\label{eq:average-degree-in-k-core}
\hat{d}=\frac{2\hat{m}}{\hat{n}}=\left(1+o\left(1\right)\right)\frac{\mu_{c}\Psi_{\ge k-1}\left(\mu_{c}\right)n}{\Psi_{\ge k}\left(\mu_{c}\right)n}=\left(1+o\left(1\right)\right)\frac{j\Pr\left[Z_{k}\left(\mu_{c}\right)=j\right]}{\Pr\left[Z_{k-1}\left(\mu_{c}\right)=j-1\right]}.
\end{equation}

There is one more fact that we will need about the joint degree distribution of vertices in the $k$-core.
 Let $\gnm$ denote the Erd\H{o}s--R\'enyi random graph model, which is the uniform distribution over all graphs with $n$ vertices and $m$ edges.
For a fixed $k$, let $\K(n,m,k)$ be the probability space of graphs  with minimum degree at least  $k$, distributed as the $k$-core of  $\gnm$, and similarly define $\K(n,c,k)$ for  $\gnc$. When  defining these graphs, we only consider the vertices in the $k$-core, so the vertex numbering is compressed into the range $[1, \nn]$, where  $\nn \le n$ is the (random) number of vertices in the $k$-core, while maintaining their order from the original graph.

For the  $\K(n,m,k)$  case, a restatement of~\cite[Corollary~1]{CW-06} gives the following approximation to the distribution of  its  degree  sequence.  Let ${\cal M}(\nn,\hat m, k)$ denote the probability space of sequences $(M_1,\ldots, M_{\nn})$ with the multinomial distribution,   with parameters $\nn$ and $2\hat m$, but conditioned upon $M_i\ge k$ for all $i$.   Let ${\cal M}(\nn,\hat m)$ denote the probability space of sequences $(X_1,\ldots, X_{\nn})$ summing to $2\mm$ with the multinomial distribution, that is, for every $(d_1, \dots, d_{\nn}) \in \D_{\nn,2\mm}$, the probability that $X_i = d_i$ for all $i$ is $(2\mm)!/(\nn^{2\mm}\prod d_i!)$. Let ${\cal M}(\nn,\hat m, k)$ denote the same probability space, but conditioned upon $X_i\ge k$ for all $i$.
%%%
\begin{prop}\label{thm:maincorfirst}
Let $k\ge 3$ and $c>c_k$ be fixed, and   $m = (1+o(1))cn/2$. Let  $H_n$  be any event in
the probability space defined by the random vector distributed as the degree
sequence of $\K(n,m,k)$.
 Suppose that whenever $\hat n$ and $\hat m$ have the asymptotic behaviour given in~\eqref{eq:number-of-vertices-in-k-core} and~\eqref{eq:number-of-edges-in-k-core} respectively, it follows that
$$
\prob_{{\cal M}(\nn,\hat m, k)} (H_n) <P_n.
$$
Then $\prob(H_n)=O( P_n)$+o(1).
\end{prop}
%%%
Next, let $\P(\nn,\mm,k)$ denote the probability space of sequences consisting of $\hat n$ independent copies of $Z_k(\lambda)$, where $\lambda$ is chosen so that
\begin{equation}\label{lambdadef}
\E \left[Z_k(\lambda)\right] = \frac{2\hat{m}}{\hat{n}}.
\end{equation}
As observed in the proof of~\cite[Lemma~1]{CW-06}, the probability that the sum of   $\nn$ copies  of $Z_k(\lambda)$ is $2\hat m$ is $\Omega(1/\sqrt{\nn})$. It follows that we may replace $\prob_{{\cal M}(\nn,\hat m, k)}$ by  $\prob_{\P(\hat n,\mm,k)}$  in the above proposition, as long as we replace $O( P_n)$ by $O(\sqrt n \cdot P_n)$. We may also replace  $\K(n,m,k)$ by   $\K(n,c,k)$ where $c=2m/n$,   using the well-known strong connection between $\gnm$ and  $\gnp$ in this case. That is, we have the following.

\begin{thm}\label{thm:maincor}
Let $k\ge 3$ and $c>c_k$ be fixed, and let  $H_n$  be any event in
the  probability space defined by the random vector distributed as the degree
sequence of $\K(n,c,k)$.
Suppose that whenever $\hat n$ and $\hat m$ have the asymptotic behaviour given in~\eqref{eq:number-of-vertices-in-k-core} and~\eqref{eq:number-of-edges-in-k-core} respectively, it follows that
$$
\prob_{\P(\hat n,\mm,k)} (H_n) <P_n.
$$
Then $\prob(H_n)=O(\sqrt n \cdot P_n)$+o(1).
\end{thm}

%%%%%%%%%%%%%%%%%%%%%%%%
\section{\label{sec:technical-proofs}Technical Proofs}

In this section we prove Theorem~\ref{thm:shatter-time} (in Section~\ref{sec:shatter-proof}),
Lemma~\ref{lem:good-tree} (in Section~\ref{sec:good-tree}) and Lemma~\ref{lem:no-bad-two-cores}
(in Section~\ref{sec:no-bad-two-cores}). We first provide some more results that will be used in the proofs of these lemmas.

    Let $\nn$ and $\mm$ denote the number of vertices and edges in $\K$, respectively.   Since in the proof of Lemma~\ref{lem:no-bad-two-cores} we work with fixed degree sequences, we wish to characterise a set of  sequences which contains the   degree sequences typical for~$\K$, and is in particular  compliant with the typical asymptotic values of~$\nn$ and~$\mm$, as well as the typical degree histogram of~$\K$, stated in Section~\ref{sec:new}.   Note that the following definition   technically and tacitly applies  to a fixed sequence of degree sequences $\dd$, one for each $n$, since it describes an asymptotic property of the degree sequence.

\begin{defn*}
A degree sequence $\dd \in \D_{\nn,2\mm}$ is \emph{proper} (with respect to the underlying parameters $n$,$k$ and $c$) if~$\nn$ and~$\mm$ satisfy~\eqref{eq:number-of-vertices-in-k-core} and~\eqref{eq:number-of-edges-in-k-core} respectively,   $\Delta(\dd) \le \log n$,
$\sum [d_i]_2 = \Theta(n)$, and the degree distribution follows the asymptotics in~\eqref{degree-asy}. In this case, we have
 for each fixed integer $j\ge 0$ that
\begin{equation*}
\left|\{i \mid d_i = j\}\right| = \begin{cases}
0, & j < k;\\
\left(1+o\left(1\right)\right)\Psi_{j}\left(\mu_c\right)n & j\ge k.
\end{cases}
\end{equation*}
 \end{defn*}

Given the degree sequence of $\K$, we perform some of the analysis in the proof of Lemma~\ref{lem:no-bad-two-cores} by using the configuration model. In order to make the results we obtain via this model applicable, we need the following immediate corollary of Theorem~\ref{thm:prob-simple}.

\begin{cor}\label{cor:prob-proper}
Let $f(n)$ be any function satisfying $f(n) \to \infty$. Then
\[
\Pr_{G \sim \mathcal{G}(n,\dd) }\left[G\in \P\right] \le f(n)\cdot\Pr_{G^* \sim \mathcal{G}^*(n,\dd)}\left[G^* \in \P\right]
\]
for any proper degree sequence $\dd \in \D_{\nn,2\mm}$ and any multigraph property $\P$.
\end{cor}

To show that we may restrict to proper degree sequences, and other kinds to be defined below, we first show that the upper tail of the sum of squared degrees is negligible in $\mathcal{G}(n,p)$, as follows. \begin{lem}\label{l:second_moment_degrees}
Let $(D_0, D_1,\ldots )$ be the degree histogram of $G\in\gnc$. Then WHP
\begin{description}
\item{(a)}
 $D_j=0$ for all $j\ge \log n$;
\item{(b)} for every $\epsilon>0$ there exists an integer $j_0$ such that  $ \sum_{j\ge j_0} [j]_2D_j/n <\epsilon$.
  \end{description}
\end{lem}
\begin{proof} Part (a) is well-known and follows from the third case in~(\ref{eq:degree-distribution-gnc}). Part (b) also follows by  standard methods, for instance as follows. First, note that we may assume that (a) holds and hence restrict the summation to $j<\log n$. Standard computations show for such $j$  that $\E [D_j]<c^jn/j!$ and the variance of $D_j$ is $O(n \log^2 n)$.  Chebyshev's inequality, together with the union bound, then  implies that WHP $D_j<\E [D_j]+n^{3/4}$ for all $j<\log n$. The result now follows, given the above bound on $\E [D_j]$.
\end{proof}
Note that if $G$ satisfies Part~$(b)$ of the lemma, then $\sum [d_i]_2 = O(n)$.
Indeed, let $j_0$ such that $ \sum_{j\ge j_0} [j]_2D_j/n < 1$. Then all vertices of $G$ of degree at least $j_0$ contribute at most $n$ to $\sum [d_i]_2$, while all other vertices contribute at most $[j_0]_2n$.
Since the bounds on degree counts of $G\in\gnc$ are also bounds for its core $\K$, an immediate consequence of Lemmas~\ref{CWnm},~\ref{CWdeg} and~\ref{l:second_moment_degrees} is the following.
\begin{cor}\label{l:proper_core}
The degree sequence of the core $\K$ is WHP proper.
\end{cor}

\medskip

The following claim, related to moments of the Poisson distribution,
is used in the proof of Theorem~\ref{thm:shatter-time}.
\begin{claim}
\label{clm:poisson-convolution}For real numbers $\mu\ge\lambda\ge0$
and integers $k\ge\ell\ge0$ we have
\[
\sum_{j=0}^{\infty}\left[j\right]_{\ell}\Psi_{j}\left(\lambda\right)\Psi_{\ge k-j}\left(\mu-\lambda\right)=\lambda^{\ell}\Psi_{\ge k-\ell}\left(\mu\right).
\]
\end{claim}

\begin{proof}
First we prove the claim for $\ell = 0$, that is
\begin{equation}\label{eq:poisson-convolution-0}
\sum_{j=0}^{\infty}\Psi_{j}\left(\lambda\right)\Psi_{\ge k-j}\left(\mu-\lambda\right)=\Psi_{\ge k}\left(\mu\right).
\end{equation}
Let $X\sim \Poisson(\lambda)$ and $Y\sim \Poisson(\mu-\lambda)$ be independent. Then $X+Y\sim\Poisson(\mu)$ and thus
\begin{align*}
\sum_{j=0}^{\infty}\Psi_{j}\left(\lambda\right)\Psi_{\ge k-j}\left(\mu-\lambda\right)
&= \sum_{j=0}^{\infty} \Pr[X=j]\Pr[Y\ge k-j]
= \sum_{j=0}^{\infty} \Pr[X=j \wedge Y\ge k-j] \\
&= \Pr[X+Y\ge k]
= \Psi_{\ge k}\left(\mu\right).
\end{align*}

Now, using the fact that $\left[j\right]_{\ell}=0$ for every $0\le j<\ell$ we get that the claim holds for all $\ell$:
\begin{align*}
\sum_{j=0}^{\infty}\left[j\right]_{\ell}\Psi_{j}\left(\lambda\right)\Psi_{\ge k-j}\left(\mu-\lambda\right)
& = \sum_{j=\ell}^{\infty}\left[j\right]_{\ell}\Psi_{j}\left(\lambda\right)\Psi_{\ge k-j}\left(\mu-\lambda\right)\\
& \stackrel{\eqref{eq:poisson-obs}}{=} \lambda^{\ell}\sum_{j=\ell}^{\infty}\Psi_{j-\ell}\left(\lambda\right)\Psi_{\ge k-j}\left(\mu-\lambda\right)\\
& = \lambda^{\ell}\sum_{j=0}^{\infty}\Psi_{j}\left(\lambda\right)\Psi_{\ge (k-\ell)-j}\left(\mu-\lambda\right)\\
& \stackrel{\eqref{eq:poisson-convolution-0}}{=} \lambda^{\ell}\Psi_{\ge k-\ell}\left(\mu\right).\qedhere
\end{align*}
\end{proof}

\medskip
The following lemma, bounding from above the probability of a truncated Poisson random variable achieving its minimum, is used in the proofs of Lemmas~\ref{lem:good-tree} and~\ref{lem:no-bad-two-cores}.
\begin{lem}
\label{lem:technical} For $k\ge3$ and $c>c_{k}$ let
\[
\delta=\delta\left(k,c\right)=\frac{1}{2}\left(1-\frac{c_{k}}{c}\right)\frac{k-2}{k-1}.
\]
Then
\[
\Pr\left[Z_{k-1}\left(\mu_{c}\right)=k-1\right]<\frac{1-2\delta}{k-1}.
\]
\end{lem}

\begin{proof}
Let
\[
F(\mu):=\frac{\Psi_{\ge k-1}(\mu)}{\Psi_{k-1}(\mu)}=\frac{1}{\Pr\left[Z_{k-1}(\mu)=k-1\right]},
\]
so we need to show that $k-1\le(1-2\delta)F(\mu_{c})$.

Let $h(\mu)=\mu/\Psi_{\ge k-1}(\mu)$, defined for all $\mu>0$. Recall
that $\beta(c)$ was determined to be the maximum solution of $x=\Psi_{\ge k-1}(cx)$,
which enables us to express $c_{k}$ as in~(\ref{eq:ck-definition}).
In terms of~$h$, we can define~$\mu_{c}$ as the maximum solution
of $h(\mu)=c$, which exists if and only if $\beta(c)>0$. We can
therefore express~$c_{k}$ again in the following way:
\[
c_{k}=\inf\{c\mid\exists\mu>0~h(\mu)=c\}=\inf\{h(\mu)\mid\mu>0\}.
\]
Clearly~$h$ is differentiable and its derivative is
\begin{align}
h'(\mu) & =\frac{1}{\Psi_{\ge k-1}(\mu)}\left(1-\mu\frac{(\Psi_{\ge k-1})'(\mu)}{\Psi_{\ge k-1}(\mu)}\right)\nonumber \\
 & =\frac{1}{\Psi_{\ge k-1}(\mu)}\left(1-\mu\frac{\Psi_{k-2}(\mu)}{\Psi_{\ge k-1}(\mu)}\right)\nonumber \\
 & =\frac{1}{\Psi_{\ge k-1}(\mu)}\left(1-(k-1)\frac{\Psi_{k-1}(\mu)}{\Psi_{\ge k-1}(\mu)}\right)\nonumber \\
 & =\frac{1}{\Psi_{\ge k-1}(\mu)}\left(1-\frac{k-1}{F(\mu)}\right)\nonumber \\
 & <\frac{1}{\Psi_{\ge k-1}(\mu)}.\label{eq:deriv}
\end{align}
Note that $h'(\mu)$ is positive if and only if $F(\mu)>k-1$, and
since~$F$ is an increasing function approaching~$1^{+}$ and~$\infty$
as~$\mu$ approaches~$0^{+}$ and~$\infty$, respectively, the
infimum~$c_{k}$ of~$h$ is actually its minimum, attained at a
unique point~$\mu_{c_{k}}$. In particular,~$h$ is increasing for
$\mu>\mu_{c_{k}}$ and since $h'(\mu_{c_{k}})=0$ we have

\begin{equation}
F\left(\mu_{c_{k}}\right)=k-1.\label{eq:zeroDeriv}
\end{equation}
By the mean value theorem (applied to~$h$) there exists some $\tilde{\mu}\in(\mu_{c_{k}},\mu_{c})$
such that
\begin{align*}
\frac{c-c_{k}}{\mu_{c}-\mu_{c_{k}}} & =\frac{h(\mu_{c})-h(\mu_{c_{k}})}{\mu_{c}-\mu_{c_{k}}}=h'(\tilde{\mu})\\
 & \stackrel{\eqref{eq:deriv}}{<}\frac{1}{\Psi_{\ge k-1}(\tilde{\mu})}<\frac{1}{\Psi_{\ge k-1}(\mu_{c_{k}})}=\frac{h(\mu_{c_{k}})}{\mu_{c_{k}}}=\frac{c_{k}}{\mu_{c_{k}}},
\end{align*}
where the second inequality is due to the monotonicity of $\Psi_{\ge k-1}$
in~$\mu$. Rearranging, we get

\begin{equation}
\mu_{c_{k}}<\left(c_{k}/c\right)\mu_{c}.\label{eq:fact1}
\end{equation}
Recall that~$F$ is an increasing function of~$\mu$, so
\begin{align}
1-2\delta & =1-\left(1-\frac{c_{k}}{c}\right)\left(1-\frac{1}{k-1}\right)\nonumber \\
 & \stackrel{\eqref{eq:zeroDeriv}}{=}1-\left(1-\frac{c_{k}}{c}\right)\left(1-\frac{1}{F\left(\mu_{c_{k}}\right)}\right)\nonumber \\
 & >1-\left(1-\frac{c_{k}}{c}\right)\left(1-\frac{1}{F\left(\mu_{c}\right)}\right)\nonumber \\
 & =\frac{1+\left(c_{k}/c\right)\left(F\left(\mu_{c}\right)-1\right)}{F\left(\mu_{c}\right)}.\label{eq:fact2}
\end{align}
Finally, observe that for every $\mu>0$
\begin{align*}
F(\mu)-1 & =\frac{\Psi_{\ge k-1}(\mu)-\Psi_{k-1}(\mu)}{\Psi_{k-1}(\mu)}=\frac{\Psi_{\ge k}(\mu)}{\Psi_{k-1}(\mu)}\\
 & =\sum_{j=k}^{\infty}e^{-\mu}\frac{\mu^{j}}{j!}\Big/\left[e^{-\mu}\frac{\mu^{k-1}}{(k-1)!}\right]\\
 & =(k-1)!\mu^{1-k}\sum_{j=1}^{\infty}\frac{\mu^{j+k-1}}{(j+k-1)!}\\
 & =\sum_{j=1}^{\infty}\frac{\mu^{j}}{\left[j+k-1\right]_{j}},
\end{align*}
and thus $F(\alpha\mu)-1<\alpha(F(\mu)-1)$ for every $0<\alpha<1$,
implying
\[
k-1\stackrel{\eqref{eq:zeroDeriv}}{=}F\left(\mu_{c_{k}}\right)\stackrel{\eqref{eq:fact1}}{<}F\left(\left(c_{k}/c\right)\mu_{c}\right)<1+\left(c_{k}/c\right)\left(F\left(\mu_{c}\right)-1\right)\stackrel{\eqref{eq:fact2}}{<}\left(1-2\delta\right)F\left(\mu_{c}\right),
\]
which establishes the lemma.
\end{proof}

\subsection{\label{sec:shatter-proof}Proof of Theorem~\ref{thm:shatter-time}}

Recall that   degree histograms were defined in Section~\ref{sec:branching-gnp}. An \emph{asymptotic degree histogram}  is a sequence $D=\left(D_{0},D_{1},\ldots\right)$
of functions $D_{j}:\mathbb{N}\to\mathbb{N}$ such that $\sum_{j=0}^{\infty}D_{j}\left(n\right)=n$
and $\sum_{j=0}^{\infty}jD_{j}\left(n\right)$ is even for all $n$.
For a given asymptotic degree histogram $D$, denote by $\Omega\left(n,D\right)$ the set of all simple graphs on $n$ vertices with degree histogram $\left(D_0(n),D_1(n),\ldots\right)$.
If $\Omega\left(n,D\right)\neq\varnothing$ for all $n\ge1$, $D$ is \emph{feasible}; in this case, let $\mathcal{G}\left(n,D\right)$ be the uniform distribution over $\Omega\left(n,D\right)$.
A feasible asymptotic degree histogram $D$ is \emph{sparse}
if $\sum_{j=0}^{\infty}jD_{j}\left(n\right)/n=\kappa_{D}+o\left(1\right)$
for some constant~$\kappa_{D}$, called the asymptotic edge density
of $\mathcal{G}\left(n,D\right)$; $D$ is \emph{well-behaved} if:
\begin{enumerate}
\item There exist constants $\left\{ \delta_{j}\right\} _{j=0}^{\infty}$
such that $\lim_{n\to\infty}D_{j}\left(n\right)/n=\delta_{j}$ for
all fixed $j\ge0$.
\item $\left\{ j\left(j-2\right)D_{j}\left(n\right)/n\right\} _{j=0}^{\infty}$
tends uniformly to $\left\{ j\left(j-2\right)\delta_{j}\right\} _{j=0}^{\infty}$.
\item $\lim_{n\to\infty}\sum_{j=0}^{\infty}j\left(j-2\right)D_{j}\left(n\right)/n$
exists, and the sum uniformly approaches the limit $$Q_{D}:=\sum_{j=0}^{\infty}j\left(j-2\right)\delta_{j}.$$
\end{enumerate}
Molloy and Reed~\cite{MR-95} showed that the sign of $Q_{D}$ WHP determines the
existence of a giant component in $\mathcal{G}\left(n,D\right)$:
\begin{lem}[{\cite[Theorem~1]{MR-95}}]
\label{lem:molloy-reed}Let $D$ be a feasible well-behaved sparse
asymptotic degree histogram, and let $\Delta\left(n\right)=\max\left\{ j\in\mathbb{N}\mid D_{j}\left(n\right)>0\right\}$.
\begin{enumerate}
\item If $\Delta\left(n\right)=o\left(n^{1/4}\right)$ and $Q_{D}>0$ then
WHP $\mathcal{G}\left(n,D\right)$ has a linear-size connected component;
\item If $\Delta\left(n\right)=o\left(n^{1/8}\right)$ and $Q_{D}<0$ then
the size of the largest connected component in $\mathcal{G}\left(n,D\right)$
is WHP $O\left(\Delta^{2}\left(n\right)\log n\right)$.
\end{enumerate}
\end{lem}

Having computed the likely degree histogram $D^{t}$ of $G_{t}$
in Claim~\ref{clm:degree-distribution-after-t-steps}, we are ready
to prove Theorem~\ref{thm:shatter-time}.
\begin{proof}[Proof of Theorem~\ref{thm:shatter-time}]
We can view $G_{t}$ as drawn from $\mathcal{G}\left(n,D^{t}\right)$,
since an iteration of the peeling process can be carried out as follows:
first expose the set of vertices $V_{t-1}=\left\{ v\mid\rho\left(v\right)=t-1\right\} $,
then expose their degrees in $G_{t-1}$; finally expose and delete
all edges incident with $V_{t-1}$. Given the degree histogram $D^{t}$,
all edges inside $G_{t}$ remain unexposed.

Implicitly, $G_{t}$ can be regarded as a sequence of random graphs, one for each $n$, and their degree histograms $D^t$ determine an asymptotic degree histogram which we denote by $(D^t)$. Although results like Claim~\ref{clm:degree-distribution-after-t-steps}  only describe events that hold WHP, they can easily be converted to statements about an asymptotic degree histogram such that the events hold for all $n$, and the asymptotic degree histogram WHP coincides with the random graph. (This can be done by altering the histogram on those values of $n$ which violate the required properties of being well behaved.) When we make statements about $(D^t)$ in the following, we assume these slight adjustments are made automatically.
Fix $t\ge 0$.  By definition $D^{t}$ is feasible and it is easy to verify that $D^{t}$
is WHP well-behaved using Claim~\ref{clm:degree-distribution-after-t-steps} for bounded degrees, together with Lemma~\ref{l:second_moment_degrees} for the unbounded degrees.  (Later we will choose a particular value of the constant $t$.)
Applying Claim~\ref{clm:poisson-convolution} with $\lambda=c\beta_{t}$, $\mu=c\beta_{t-1}$ and $\ell=1$, we get that $D^{t}$ is also sparse, with asymptotic edge density
\[
\kappa_{t}:=\kappa_{D^{t}}=\sum_{j=0}^{\infty}j\delta_{j}^{t}=\sum_{j=0}^{\infty}j\Psi_{j}\left(c\beta_{t}\right)\Psi_{\ge k-j}\left(c\beta_{t-1}-c\beta_{t}\right)=c\beta_{t}\Psi_{\ge k-1}\left(c\beta_{t-1}\right)=c\beta_{t}^{2}.
\]
We now bound the parameter $Q_t := Q_{D^t}$, by another application of Claim~\ref{clm:poisson-convolution} with $\lambda=c\beta_{t}$ and $\mu=c\beta_{t-1}$, but this time with $\ell=2$. We get
\[
\sum_{j=0}^{\infty}\left[j\right]_2\delta_j^t
=\sum_{j=0}^{\infty}\left[j\right]_2\Psi_j\left(c\beta_t\right)\Psi_{\ge k-j}\left(c\beta_{t-1}-c\beta_t\right)
=\left(c\beta_t\right)^2\Psi_{\ge k-2}\left(c\beta_{t-1}\right)
=c\kappa_t\Psi_{\ge k-2}\left(c\beta_{t-1}\right),
\]
hence
\begin{align*}
Q_{t} & =\sum_{j=0}^{\infty}j\left(j-2\right)\delta_{j}^{t}=\sum_{j=0}^{\infty}\left(\left[j\right]_{2}-j\right)\delta_{j}^{t}\\
 & =c\kappa_{t}\Psi_{\ge k-2}\left(c\beta_{t-1}\right)-\kappa_{t}\\
 & =\left(c\Psi_{\ge k-2}\left(c\beta_{t-1}\right)-1\right)\kappa_{t}.
\end{align*}
The decreasing sequence $\left(\beta_{t}\right)_{t\ge0}$ converges
to $\beta=0$ in the subcritical regime, and $x\mapsto c\Psi_{\ge k-2}\left(cx\right)$
is a continuous function, so
\[
\lim_{t\to\infty}c\Psi_{\ge k-2}\left(c\beta_{t}\right)=c\Psi_{\ge k-2}\left(c\beta\right)=0.
\]
In particular, there exists some constant $t^{\dagger}$ such that
$c\Psi_{\ge k}\left(c\beta_{t^{\dagger}}\right)<1$, implying $Q_{t^{\dagger}}<0$.
Moreover, $\Delta\left(G_{t^{\dagger}}\right)\le\Delta\left(G_{0}\right)=O\left(\log n/\log\log n\right)$
by~(\ref{eq:degree-distribution-gnc}), and we can finally apply
the second part of Lemma~\ref{lem:molloy-reed} and complete the
proof.
\end{proof}

\begin{rem*}
By~\cite[Lemma~6]{JMT-16} we have $t^{\dagger}=\Theta\left(1/\sqrt{c_{k}-c}\right)$.
\end{rem*}

%%%%%%%%%%%%%%%%%%%%%%%
\subsection{\label{sec:good-tree}Proof of Lemma~\ref{lem:good-tree}: the existence
of a $\left(\log^{2}\log n,L\right)$-tree in $\K$}

We prove Lemma~\ref{lem:good-tree} for $L:=\left\lceil \log_{1-\delta}\left(\delta^{2}/(2\mu_{c})\right)\right\rceil $, where $\delta  <  1/2$ is the constant from Lemma~\ref{lem:technical}. Throughout this subsection we let $N  = \left\lceil \log^{2}\log n\right\rceil$ and $p_j  = \Pr[Z_{k-1}(\mu_{c})=j-1]$.
In addition, we set $C = (1-\delta)^2/(1-2\delta)>1$, and let $d_0$ be the minimal integer satisfying $1-\sum_{i \le d_0} p_i/C \le \delta^2/(2\mu_{c})$. Note that $L,C$ and $d_0$ are all constants depending only on $k$ and $c$.
\medskip

We consider an exploration process in the  $k$-core $\K$, attempting to reveal an $(N,L)$-tree in it, but instead of analysing  the exploration process on $\K$ itself, we condition on  it having  a proper degree sequence $\dd$ and apply the exploration process to the configuration model for the sequence~$\dd$. In view of Lemma~\ref{CWnm} and Corollary~\ref{l:proper_core}, Theorem~\ref{thm:prob-simple} implies that it  is enough  to show that the multigraph of this configuration model WHP contains an $(N,L)$-tree,  where the convergence implicit in WHP  is  uniform over all proper  degree sequences  $\dd$.

 The exploration starts with an arbitrary vertex $v_0$  in this configuration model, and explores its $(2N)$-neighbourhood  in DFS manner. In each exploration step, an unmatched half-edge, say $x$, belonging to an exposed vertex at distance at most $2N-1$ from $v_0$, is matched to some other half-edge, say $y$, chosen uniformly at random from the set of all unmatched half-edges. We refer to the vertex containing $y$ as the \emph{next encountered vertex}. The selection of $x$ in each step is arbitrary among those in the vertex currently being treated by the DFS algorithm.    Initially, $v_0$ is the only exposed vertex, and whenever a new half-edge is being matched, its vertex  (i.e.\ the next encountered vertex) becomes exposed. Unless that vertex was already exposed, the new edge and vertex are added to the growing DFS tree.

Let $T$ denote the tree resulting from the exploration described above. The root of T is $v_0$, and all other vertices of $T$ have distance at most $2N$ from $v_0$ in $T$.
We assign each vertex in $T$
a \emph{type} from $\left\{ 0,1,\ldots,L+1\right\} $ in the following
manner. The type assignment for a vertex $u$ is performed at  the point  when the DFS  algorithm has finished fully exploring the  subtree of $T$ below $u$ and looks to move  back to the parent of $u$ (or terminate, if $u = v_0$). First, if any  back-edge  has been encountered  up to this point in  the  exploration  process, $u$ is assigned type $L+1$.  If no back-edge has been encountered, the following  rules are applied.
If  $u$ is a leaf, i.e. in level $2N$, its  type is set  to $0$.
Otherwise, $u$ is in level $i<2N-1$, and all its children have been assigned types already; denote by $S(u)$ the set of its children of type less than $L$.
Note that $d_T(u) = d_\K(u)$ in this case, so we can omit the subscript, and set
\[
\type\left(u\right)=\begin{cases}
0, & \left|S\left(u\right)\right|\ge k\quad\ \  \mbox{ and } d(u)\le d_0;\\
1+\max\left\{ \type\left(v\right)\mid v\in S\left(u\right)\right\} , & \left|S\left(u\right)\right|=k-1\mbox{ and } d(u)\le d_0;\\
L, & \left|S\left(u\right)\right|<k-1\ \ \,  \mbox{ or } d(u)>d_0.
\end{cases}
\]

For $v\in V(T)$,  let $T(v)$ denote  the subtree  of  $T$  consisting of  $v$ and  all  its  descendants. Given the types of vertices as defined above, let $T^*(v)$ denote the result of removing from  $T(v) $  all subtrees
rooted at vertices of type $L$  or  $L+1$. Then  for every vertex $u$ in
 $T^*(v)$, the number of vertices in a longest $k$-light tree path   originating at $u$    is exactly $\type(u) <L$. If $\type(u) = 0$ it simply means that no such paths exist as $u$  itself  is not $k$-light. In particular, we have the following.
\begin{obs}\label{obs:type}
Let $v\in V(T)$  at level $i$. If  $\type(v)<L$ then $T^*(v)$ is a  $(2N-i,L)$-tree.
\end{obs}

Let us now take a closer look at the process of matching half-edges. When the first random half-edge is chosen, the probability that its vertex $u$ has degree $j$ is weighted by a multiplicative factor of $j$ (so-called ``degree-biased'' selection). Hence, for any fixed $j \ge k$ and a proper degree sequence, we have by~\eqref{degree-asy} that the probability that $u$ has degree $j$ is
\begin{equation}\label{eq:deg-biased}
(1+o(1))j\Pr[Z_{k}\left(\mu_{c}\right)=j]\nn/2\mm \stackrel{\eqref{eq:average-degree-in-k-core}}{=} (1+o(1))p_j.
\end{equation}
As the exploration carries on, the degree sequence of the unmatched half-edges does not represent the degrees of the vertices any more, but their ``remaining'' degrees, that is, the number of unmatched half-edges incident with each vertex at that point. Of course, this distinction only applies to the exposed vertices. Additionally, this degree sequence contains values smaller than $k$. To handle these subtleties, we define a new class of sequences, closely related to proper sequences.

For a constant $\eta > 1$, a degree sequence $\dd \in \D_{\nn,2\mm}$ is \emph{$\eta$-normal} (with respect to $n,k,c$ and $d_0$), if $n \ge \nn \ge \Psi_{\ge k}\left(\mu_{c}\right)n/2$ and $\mm \ge \mu_{c}\Psi_{\ge k-1}\left(\mu_{c}\right)n/4$, its maximum degree is at most $\log n$, and for every $j \le d_0$ the degree-biased probability that a half-edge selected uniformly at random belongs to a vertex of (remaining) degree $j$ is between $p_j/\eta$ and $\eta p_j$. Note that in particular, for every $j \le d_0$, the number of  `$j$'  entries in any $\eta$-normal sequence is at least $2\mm p_j/(j\eta) = \Omega(n)$. It is immediate to see that for any fixed $\eta > 1$, every proper sequence is $\eta$-normal for $n$ sufficiently large.

We are finally ready for the main part of the proof.
Recall the definition of $C$ from the beginning of this subsection, and let $C' = (C+1)/2 > 1$.
\begin{claim}\label{clm:pruning}
Assume that $\dd$ is a $C'$-normal sequence and consider any moment during the exploration process when a random half-edge  is about to be chosen.    Conditional upon the  exploration so far, the probability that the    next encountered vertex will  eventually have  type $L$ is at most $\delta/\mu_{c}$.
\end{claim}

We first show  why    Claim~\ref{clm:pruning} implies Lemma~\ref{lem:good-tree}.
Assume $\dd$ is $C'$-normal and  consider any moment at which an unmatched half-edge incident with a vertex at  level $N-1$ in $T$  is about to be treated, i.e., the next encountered vertex will belong to level $N$ unless already exposed. At this point the types of all vertices so far encountered at level $N$ have been assigned.   We may therefore apply
Claim~\ref{clm:pruning}
to deduce that, conditional on the labels  of the previously encountered vertices in level $N$, the probability that the next one receives type  $L$  is at most $ \delta/\mu_c $. By coupling this process with a sequence of independent Bernoulli trials each with parameter $\delta/\mu_c$, we conclude that, for any $t>0$,  the probability that  at least $t$ vertices are encountered at level $N$ and all are given type  $L$ is at most $(\delta/\mu_c)^t$.

Since    $\dd$ is $C'$-normal,  the degrees are all at most $\log n$, and so the number of vertices reached in the exploration is at most $(\log n)^{2N}= \exp\big(O(\log^{3}\log n)\big)=o(n^{1/3})$.  Consequently, since there are $\Theta(n)$ half-edges altogether, each step of the exploration process chooses an exposed vertex with probability $o(n^{-2/3}\log n)$. Thus, the probability that at least one of the $o(n^{1/3}\log n)$ steps encounters a back-edge is $o(1)$. Since all vertices of  $\K$ have degree at least  $k$, it follows that WHP, $T$ has
at  least $ (k-1)^{N} = \omega(1)$ vertices at level $N$. From the previous paragraph, the probability that these are all assigned  type  $L$ is $o(1)$. As there are WHP no back-edges, this implies that WHP some vertex $v$ at level $N$ receives a  type less than $L$.  By Observation~\ref{obs:type},   this  event implies  that $T^*(v)$ is an $(N,L)$-tree.

One can easily check that the convergence in the above WHP statements is uniform over all $C'$-normal degree sequences $\dd$. Since WHP   $\K$ has a $C'$-normal degree sequence,  Lemma~\ref{lem:good-tree} follows, and it only remains to prove the claim.

\begin{proof}[Proof of Claim~\ref{clm:pruning}]
For the given degree sequence $\dd = (d_1,\dots,d_{\nn}) \in \D_{\nn,2\mm}$,
let $R(\dd)$ be the set of all sequences $\dd' = (d_1',\dots,d_{\nn}') \in \D_{\nn,2\mm'}$, such that $d_i \ge d_i'$ for all $i$, and $m-m' \le n^{1/3}\log n$. By the arguments above, at every step of the exploration, the degree sequence of the unmatched half-edges in $\K$ is some element of $R(\dd)$.

Now let $\dd' \in R(\dd)$, and for every $j \le d_0$, let $n_j$ and $n_j'$ denote the number of vertices of degree $j$ in $\dd$ and in $\dd'$, respectively. Since $n_j = \Theta(n)$, and since $d_i \neq d_i'$ for $O(n^{1/3}\log n)$ coordinates, we have $n_j' = (1+o(1))n_j$. Similarly, $m' = (1+o(1))m$, and thus $jn_j'/(2m') = (1+o(1))jn_j/(2m)$. In short, for every $\dd' \in R(\dd)$  and $j \le d_0$, the probability of choosing a half-edge belonging to a vertex of (remaining) degree $j$, asymptotically equals the probability of the same event for $\dd$. Since, in addition, the probability that a randomly selected half-edge belongs to an already exposed vertex is $o(1)$, and since $C > C'$, we can state the following.

\begin{obs}\label{obs:Cbound}
At every step of the exploration,  conditioning upon the exploration steps taken previously,  the probability that the next encountered vertex is unexposed and has degree~$j \le d_0$ is bounded between $p_j/C$ and $Cp_j$.
\end{obs}
%\medskip
%\bigskip

To complement this observation, consider any moment during the exploration, let $S$ denote the exploration sequence up to that point, and let $p_{>d_0} (S)$ denote the probability that the next encountered vertex  will have degree larger than $d_0$, conditional on $S$. Let $p_{>d_0}$ denote the maximum of $p_{>d_0} (S)$, taken over all possible (partial) exploration sequences $S$. Then by Observation~\ref{obs:Cbound} and the definition of $d_0$ (at the beginning of this subsection), we have ${p_{>d_0} \le \delta^2/(2\mu_{c})}$.
\medskip

Before proceeding with the proof of the claim we introduce more terminology.  We say a  vertex $v$ of $T$  has height  $h$  if it is in level $2N-h$,  i.e.\ $T(v)$ has height $h$.
The height of an unmatched half-edge belonging to an exposed vertex is the same as the height of that vertex. 

Similarly to the definition of $p_{>d_0}$, for every $1 \le h \le 2N $ and $0\le i\le L$, let $P_{i,h}$ denote the maximum, over all   possible exploration sequences $S$ up to  any  step  in which a half-edge  at height   $h$   is being matched, of the probability that  the next encountered vertex  will be assigned type $i$, conditional on $S$. Note that this next encountered
vertex is at height $h-1$ unless it was already exposed.

We prove the claim by showing that the following hold for every $1 \le h \le 2N $:
\begin{enumerate}
\item[$(a)$] $P_{L,h}\le\delta/\mu_{c}$;
\item[$(b)$] $P_{i,h}\le(1-\delta)^{i}$ for every $0\le i<L$.
\end{enumerate}
%Recall that  if  a back-edge occurs before assigning the type  of a vertex, it is given type $L+1$. If no back-edge has yet occurred when a vertex $u$ is assigned type, then $d_T(u) = d_\K(u)$, so there is no need for distinction.
%

Recall that  if  a back-edge occurs before assigning the type  of a vertex $u$, it is given type $L+1$, and that otherwise, and if $u$ is also not a leaf, then $d_T(u) = d_\K(u)$ and we simply refer to the degree of $u$ with no specification.  

Observe that $(b)$ holds trivially for $i=0$ (for every $h$), so from now on we assume for simplicity $i>0$. Since the only positive type a leaf can be assigned is $L+1$, there is nothing to prove for $h = 1$.
We now prove $(a)$ and $(b)$ for $h > 1$ by induction on $h$, beginning with $h = 2$. 
\medskip

Assume that no back-edge has yet occurred when a vertex $u$ at height 1 is assigned its type. Then every child of $u$ is a leaf of type 0, and there are therefore exactly three options. If $d(u) > d_0$ then $\type(u) = L$; if $d(u) = k$, i.e., $u$ has $k-1$ children in $T$, then $\type(u) = 1$; otherwise, $\type(u) = 0$.
So when a half-edge at height 2 is being matched, the probability that the next encountered vertex will eventually have type 1 or $L$ is bounded from above by the probability that the vertex will have degree $k$ or larger than $d_0$, respectively.
We therefore immediately obtain $(a)$ since
\[
P_{L,2} \le p_{>d_0} \le \delta^2/(2\mu_{c}) \le \delta/\mu_{c}.
\]
As for $(b)$, we only have to show that $P_{1,2} \le 1-\delta $, and this is true since by the above argument, Observation~\ref{obs:Cbound} and Lemma~\ref{lem:technical} we have
\[
P_{1,2}
\le Cp_k
<C\frac{1-2\delta}{k-1}
=\frac{(1-\delta)^2}{k-1}
<1-\delta.
\]
Assume now that $(a)$ and $(b)$ hold for $1 < h - 1 < 2N$; we prove $(a)$ and $(b)$ for $h$.

\medskip
 Observe that  if a vertex $v$ at height~$h -1$   with degree~$j$ is assigned type
$0<i<L$ (implying in particular that $j \le d_0$ by definition of type), then no  back-edge has yet been encountered, exactly $j-k$ of $v$'s $j-1$ children must be assigned type $L$,
and the maximum type among its other $k-1$ children must be exactly $i-1$.

 Now consider any moment when a half-edge at height $h$ is matched, and let $w$ denote the next encountered vertex. By Observation~\ref{obs:Cbound}, the probability that $w$ will be an unexposed vertex with degree $j$ is at most $Cp_j$.
Since the exploration is DFS, the bounds $P_{i,h-1}$ (for any $0 \le i \le L$) can be applied to each of the children of $w$ consecutively  and conditionally. Thus at any given step in which a half-edge at height~$h$ is being matched,  the probability that  the next  encountered vertex
will have degree $j$ and  type $0<i<L$, conditional on the previous history of the process, can be bounded from above  by
\[
f_{i,h}(j) := Cp_{j}\binom{j-1}{j-k}(P_{L,h-1})^{j-k}(k-1)P_{i-1,h-1}.
\]

In the same manner, consider the event that a  vertex $v$ at height $h-1$  with degree $j \le d_0$  is assigned
type $L$. (Recall that a vertex will also be assigned type $L$ if it has degree larger than $d_0$, which happens with probability at most $p_{> d_0}$.) Then no back-edge has yet occurred, and either exactly $j-k$ of $v$'s children are of
type $L$  and there is at least one child of type $L-1$, or there
are at least $j-k+1$ children of $v$ of type $L$. So at any moment in which a half-edge at height~$h$ is being matched,  we can bound
from above the probability that the next  encountered   vertex   will have degree $j \le d_0$ and  be assigned type $L$, conditional on the previous history of the process, by
\[
f_{h}(j) := Cp_{j}\binom{j-1}{j-k}(P_{L,h-1})^{j-k}(k-1)(P_{L,h-1}+P_{L-1,h-1}).
\]
%In addition, a vertex is also assigned type $L$ if it has degree larger than $d_0$, and the (degree-biased) probability of this --- denoted by $p_{> d_0}$ --- is at most $\delta^2/(2\mu_{c})$ by definition of $d_0$.
%
\medskip
Next, we observe that for every $0<i<L$ and for every $k \le j < d_0$ the following holds:
\begin{equation}\label{eq:quotient}
\frac{f_{i,h}(j+1)}{f_{i,h}(j)}
=\frac{f_{h}(j+1)}{f_{h}(j)}
=\frac{p_{j+1}}{p_{j}}\cdot \frac{\binom{j}{j-k+1}}{\binom{j-1}{j-k}}\cdot P_{L,h-1}
=\frac{\mu_{c}}{j}\cdot\frac{j}{j-k+1}\cdot P_{L,h-1}
\le\delta,
\end{equation}
where the last inequality follows from the induction hypothesis.
We now use Lemma~\ref{lem:technical} and the induction hypothesis to get
\[
f_{i,h}(k)
=Cp_{k}(k-1)P_{i-1,h-1}
< C(1-2\delta)\left(1-\delta\right)^{i-1}
=\left(1-\delta\right)^{i+1},
\]
which together with~(\ref{eq:quotient}) yields
\[
P_{i,h}
\le \sum_{j = k}^ {d_0} f_{i,h}(j)
\le \sum_{j = k}^ {d_0} f_{i,h}(k)\delta^{j-k}
\ <\ \frac{f_{i,h}(k)}{1-\delta}<\left(1-\delta\right)^i,
\]
establishing $(b)$.
To prove $\left(a\right)$, we similarly use
\[
f_{h}(k)=Cp_{k}(k-1)(P_{L,h-1}+P_{L-1,h-1})<\left(1-\delta\right)^2 \left(\frac{\delta}{\mu_{c}}+\left(1-\delta\right)^{L-1}\right)
\]
and~(\ref{eq:quotient}) to obtain

\begin{eqnarray*}
\sum_{j = k}^ {d_0}f_{h}(j)
 & \le & \sum_{j = k}^ {d_0}f_{h}(k)\delta^{j-k}\ <\ \frac{f_{h}(k)}{1-\delta}\\
 & < & \left(1-\delta\right)\frac{\delta}{\mu_{c}}+\left(1-\delta\right)^{L}
 =\frac{\delta}{\mu_{c}}-\left(\frac{\delta^{2}}{\mu_{c}}-\left(1-\delta\right)^{L}\right)\le\frac{\delta}{\mu_{c}} - \frac{\delta^2}{2\mu_{c}},
\end{eqnarray*}
where the last inequality holds by the choice of $L$. We conclude that (a) holds by using the bound on $p_{> d_0}$ and the fact that
\[
P_{L,h} \le p_{> d_0} + \sum_{j = k}^ {d_0}f_{h}(j).\qedhere
\]
\end{proof}

\subsection{\label{sec:no-bad-two-cores}Proof of Lemma~\ref{lem:no-bad-two-cores}:
high excess, small boundary 2-cores are WHP linear}

 Throughout this subsection we   denote by $\nn$ and $\mm$ the number of vertices and edges, respectively, in~$\K$, the $k$-core of $G$, and denote by $\hd = 2\mm/\nn$ the average degree in~$\K$. For a subset $U \subset V(\K)$ we write $t=\left|V\left(U\right)\right|$, $s=\left|E\left(U\right)\right|$ and $r=\xs\left(U\right)=s-t$. We refer to any graph of minimum degree at least 2 as a\emph{ 2-core}. A 2-core $C\subset\mathcal{K}$ is \emph{bad} if $C$ has large excess
$\xs\left(C\right)\ge\log^{4}n$, small boundary $\left|\partial C\right|\le\left(k-2\right)V\left(C\right)$,
and small size $\left|V\left(C\right)\right|<\eps n$. Lemma~\ref{lem:no-bad-two-cores}
claims that WHP $\mathcal{K}$ has no bad 2-cores for some constant
$\eps=\eps\left(k,c\right)>0$.
We will consider  bad 2-cores in $\K$ with exactly $t$ vertices and $s$ edges for pairs $(t,s) \in \mathcal I$, where
\[
\mathcal{I}=\left\{\left(t,s\right)\mid t<\eps n \text{~~and~~} t+\log^{4}n\le s\le\binom{t}{2}\right\}.
\]
Observe that in particular $\log^{4}n < \binom{t}{2}<t^{2}$, and so from now on we assume $t>\log^{2}n$.

\medskip
Recall the constant $\delta$ from Lemma~\ref{lem:technical}, and let $\delta_{1}<1/e$ be  constant sufficiently small that $(1-\delta/2)\delta_{1}^{-4\delta_{1}}<1-\delta/4$. Such a $\delta_{1}$ exists since $x^{-x}\rightarrow1^{+}$ as $x\rightarrow0^{+}$.
In the proof we separate potential bad 2-cores into two classes: dense (i.e.,
with excess $r\ge\delta_{1}t$) and sparse (i.e., with excess $r\le\delta_{1}t$),
and partition $\mathcal{I}$ into $\mathcal{I}=\mathcal{I}_{\dns}\cup\mathcal{I}_{\sprs}$
accordingly.

We are now finally ready to define $\eps$. Since the discussion is
restricted to the $k$-core $\mathcal{K}$, which is WHP of linear
size $\hat{n} = (1+o(1))\Psi_{\ge k}(\mu_{c})n$, it will be convenient to define a constant
\[
\eps_{1}=\min\left\{(e^2c)^{-1-1/\delta_{1}},\frac{\delta}{2-\delta},\frac{1}{1+2e^{4}k^{6}}\right\}
\]
and set $\eps=\eps_{1}\Psi_{\ge k}\left(\mu_{c}\right)/2$.

\subsubsection{Dense 2-cores: $r\ge\delta_{1}t$}
We show that WHP, not only $\K$ does not contain dense bad 2-cores, but $G$ does not contain \emph{any} dense subgraphs of relevant size, without further restrictions. Clearly, it therefore suffices to only consider the case $r = \delta_1t$. So, for $\log^2n < t < \eps n$, let $N_t$ denote the expected number of subgraphs of $G$ with $t$ vertices and $(1+\delta_1)t$ edges. Then

\begin{align*}
N_t & \le \binom {n}{t} \binom {\binom t2}{(1+\delta_1)t} \left(\frac{c}{n}\right)^{(1+\delta_1)t} \\
& \le \left(\frac{en}{t}\right)^{t}\left(\frac{et^2/2}{(1+\delta_1)t}\right)^{(1+\delta_1)t} \left(\frac{c}{n}\right)^{(1+\delta_1)t} \\
& \le \left(\frac tn \right)^{\delta_1t} \left(\frac{e^2c}{2}\right)^{(1+\delta_1)t} \\
& \le \left(\eps\left(e^2c\right)^{1+1/\delta_1}\right)^{\delta_1t}2^{-t} \\
& \le 2^{-t},
\end{align*}
where the last inequality is due to the fact that $\eps<\eps_{1}$
and the definition of $\eps_{1}$. Summing over $t$, we get
$$\sum_{t=\log^2n}^{\eps n} N_t \le \sum_{t=\log^2n}^{\eps n} 2^{-t} \le n2^{-\log^2n} = o(1),$$
implying that WHP $G$ does not contain any subgraph with $t$ vertices and $s$ edges for any $(t,s)\in\cI_{\dns}$.

\subsubsection{Sparse 2-cores: $r\le\delta_{1}t$}

By using Corollary~\ref{l:proper_core}, we  will be able to restrict to $k$-cores with proper degree sequences. Of course, ``the degree sequence is proper'' is not an event but an asymptotic statement. Strictly, what Corollary~\ref{l:proper_core} means is that there is a concrete specification of the asymptotic bounds in the definition of ``proper'' such that such bounds hold WHP for the degree sequence of the $k$-core. When we refer to the event that the sequence is proper below, we mean the event that a set of such bounds hold.

In view of Theorem~\ref{thm:maincor}, we  consider a sequence $\dd$ of $\hat n$ independent copies of $Z_k(\lambda)$ where $\lambda$ is determined by~\eqref{lambdadef}. By Lemma~\ref{CWnm} we only need to  consider $\nn = \left(1+o\left(1\right)\right)\Psi_{\ge k}\left(\mu_{c}\right)n$ and $\mm = \left(1+o\left(1\right)\right)\mu_{c}\Psi_{\ge k-1}\left(\mu_{c}\right)n/2$, and consequently, the estimation of $\hd$ given in~\eqref{eq:average-degree-in-k-core} holds.   Since $\E \left[Z_k(\lambda)\right] = \lambda \Psi_{\ge k-1}(\lambda) / \Psi_{\ge k}(\lambda)$, it follows by the definition of $\lambda$ and Lemma~\ref{CWnm}  that   ${\lambda = (1+o(1))\mu_c}$.

Let $A_n$ denote the event that $(i)$  $\dd$  is a proper sequence,  and $(ii)$    a random $k$-core $\K$ with degree sequence $\dd$ has probability at least $1/n$ of containing a   2-core with parameters $(t,s)\in \mathcal{I}_{\sprs}$. Implicitly, this event is contained in the event that the sum of components of $\dd$ is even. Note that the restriction $t < \eps n$ and the definition of $\eps$ imply that $t < \eps_1\nn$.

\medskip
In this subsection we make use of two types of degree sequences; proper degree sequences $\dd \in \D_{\nn,2\mm}$ for $\K$, and degree sequences $\dd \in \mathbb{N}^{t}$ for subsets of $V(\K)$ of size $t$. In order to distinguish between these two types we write either $\ddn$ or $\ddt$, respectively. When referring to a subset $U\subset V(\K)$ of size $t$, we use $u_{1},\dots,u_{t}$ to denote its vertices, even when this is not written explicitly.

Since estimating the expected number of sparse 2-cores involves some
tedious calculations, we make them in several steps. We begin with
a few bounds which are given without context at this moment and will
be useful later. First, for a given $\ddt\in \mathbb{N}^{t}$, by using simple combinatorial
identities and by letting $r=s-t$ and $h_{i}'=h_{i}-2$ for every $i$, we have

\begin{eqnarray}
\sum_{\substack{h_{1},\ldots,h_{t}\ge2\\
\sum h_{i}=2s
}
}\hspace{4pt}\prod_{i=1}^{t}\binom{d_{i}}{h_{i}} & = & \sum_{\substack{h_{1},\ldots,h_{t}\ge2\\
\sum h_{i}=2s
}
}\hspace{4pt}\prod_{i=1}^{t}\binom{d_{i}-2}{h_{i}-2}\frac{[d_{i}]_{2}}{[h_{i}]_{2}}\nonumber \\
 & \le & \sum_{\substack{h_{1}',\ldots,h_{t}'\ge0\\
\sum h_{i}'=2r
}
}\hspace{4pt}\prod_{i=1}^{t}\binom{d_{i}-2}{h_{i}'} [d_{i}]_{2}/2\nonumber \\
 & = & 2^{-t}\binom{\sum_{i=1}^{t}(d_{i}-2)}{2r}\prod_{i=1}^{t}[d_{i}]_{2}~.\label{eq:shortCalc}
\end{eqnarray}

Second, for $j\le2r$ let $\widetilde\D_{j}=\left\{ \ddt \in \D_{t,kt+j} \mid \forall i\,d_{i}\ge k\right\}$, and note that $\left|\widetilde\D_{j}\right|=\binom{j+t-1}{t-1}$ and that
\begin{equation}\label{eq:sum-Dj}
\sum_{j=0}^{2r}\binom{j+t-1}{t-1} = \binom{2r+t}{t} \le \binom{2t}{2r}\le \left(\frac{et}{r}\right)^{2r}.
\end{equation}
Furthermore, at least $t-j$ entries in every $\ddt\in \widetilde\D_{j}$ equal~$k$,
and $\prod_{i=1}^{t}[d_{i}]_{2}$ is maximised when the entries in
$\ddt\in \widetilde\D_{j}$ are as equal as possible; that is, when~$j$ of them
equal~$k+1$. Recall that the distribution of a single component of $\dd$ is $Z_{k}(\lambda)$, abbreviated to $Z$ for the
remainder of this subsection. Since     $\lambda = (1+o(1))\mu_c$ we may estimate the probabilities in the distribution of $Z$  asymptotically by using $Z_{k}(\mu_c)$. By Lemma~\ref{lem:technical}, for $n$ sufficiently large we have
\begin{eqnarray}
\left[k\right]_{2}\Pr\left[Z=k\right] & \stackrel{\eqref{eq:average-degree-in-k-core}}{=} & \left(1+o\left(1\right)\right)\left[k\right]_{2}\frac{\hat{d}}{k}\Pr\left[Z_{k-1}\left(\mu_{c}\right)=k-1\right]\nonumber\\
& \le & \left(1+o\left(1\right)\right)\left(1-2\delta\right)\hat{d}<\left(1-\delta\right)\hat{d}.\label{eq:fact3}
\end{eqnarray}
Hence,
\begin{eqnarray}
\sum_{\substack{d_{1},\ldots,d_{t}\ge k\\
\sum d_{i}\le kt+2r
}
}\prod_{i=1}^{t}[d_{i}]_{2}\Pr\left[Z=d_{i}\right] & = & \sum_{j=0}^{2r}\hspace{4pt}\sum_{\ddt\in \widetilde\D_{j}}\hspace{4pt}\prod_{i=1}^{t}[d_{i}]_{2}\Pr\left[Z=d_{i}\right]\nonumber \\
 & \le & \sum_{j=0}^{2r}\left|\widetilde\D_{j}\right|\left([k+1]_{2}\right)^{j}\left([k]_{2}\right)^{t-j}\Pr\left[Z=k\right]^{t-j}\nonumber \\
 & \stackrel{\eqref{eq:fact3}}{\le} & \left([k+1]_{2}\right)^{2r}\sum_{j=0}^{2r}\left|\widetilde\D_{j}\right|\left((1-\delta)\hd\right)^{t-j}\nonumber \\
 & \le & \left((1-\delta)\hd\right)^{t}\left(2k^{2}\right)^{2r}\sum_{j=0}^{2r}\binom{j+t-1}{t-1}\nonumber \\
 & \stackrel{\eqref{eq:sum-Dj}}{\le} & \left((1-\delta)\hd\right)^{t}\left(\frac{2ek^{2}t}{r}\right)^{2r}.\label{eq:longCalc}
\end{eqnarray}

We now wish to bound the probability that a given set of vertices induces a sparse bad 2-core in a random $k$-core $\K$ with a given degree sequence, and this is where the configuration model becomes useful. For a given degree sequence $\ddn \in \D_{\nn,2\mm}$ and for positive integers $t < \nn$ and $s<\mm$, let us count how many configurations $F$ for $\ddn$ yield a multigraph $H = H(F)$ such that $H[U]$ is a 2-core with $s$ edges, where $U = \{u_{1},\dots,u_{t}\}$ is the set of the first $t$ vertices in the sequence. First we have to choose a degree sequence for $H[U]$, that is, choose $\hh = \{h_{1},\dots,h_{t}\} \in \D_{t,2s}$ such that $d_{H[U]}(u_i) = h_i \ge 2$ for every $i$. Given $\hh$, there are $\prod_{i=1}^{t}\binom{d_{i}}{h_{i}}$ possibilities to determine for each $u_{i}$ which $h_{i}$ of its $d_{i}$ half-edges go inside $U$ (while the rest go outside). Finally, there are $(2s-1)!!$ configurations for $H[U]$, and $(2\mm-2s-1)!!$ configurations for the rest of $H$. It follows that if $F$ is chosen uniform
 ly at random from all $(2\mm-1)!!$ possible configurations for~$\ddn$ then the probability of $H[U]$ being a 2-core with $s$ edges is at most
\begin{equation}\label{eq:prob-config-dense}
\sum_{\substack{h_{1},\ldots,h_{t}\ge2\\\sum h_{i}=2s}}\hspace{4pt}\prod_{i=1}^{t}\binom{d_{i}}{h_{i}} \frac{(2s-1)!!(2\mm-2s-1)!!}{(2\mm-1)!!}
\stackrel{\eqref{eq:shortCalc}}{\le} \frac{2^{s}s!}{\left(2\mm-2s\right)^{s}} \cdot2^{-t}\binom{\sum_{i=1}^{t}(d_{i}-2)}{2r}\prod_{i=1}^{t}[d_{i}]_{2}~.
\end{equation}
Note that this bound does not depend on $d_{t+1},\dots,d_{\nn}$.

Now, for integers $t$ and $s$, an arbitrary subset $U\subset V(\K)$ of size $t$,
and a degree sequence $\ddt\in\mathbb{N}^{t}$, denote by $P(U,\ddt,s)$
the probability that $\K[U]$ is a bad 2-core with $s$ edges,
where $\K$ is as in the definition of the event $A_n$ but conditioned upon satisfying $d_\K(u_i) = d_i$ for every $1 \le i \le t$. By the condition on the boundary
of bad 2-cores, we only need to consider sequences $\ddt$ such that $\sum_{i=1}^{t}d_{i}\le(k-2)t+2s=kt+2r$,
and in particular $\sum_{i=1}^{t}(d_{i}-2)\le kt$, since $r<t$.
Recall that $2\mm=\hd\nn$ and note that for sparse 2-cores we have
\begin{equation}
2s=2(t+r)\le3t\le kt\le\hd\eps_{1}\nn.\label{eq:two-ess}
\end{equation}
Using~\eqref{eq:prob-config-dense}, and Corollary~\ref{cor:prob-proper} with $2^r$ standing for $f(n)$, we can bound $P(U,\ddt,s)$ from above by
\begin{eqnarray}
P(U,\ddt,s) & \le & 2^r\frac{2^{s}s!}{\left(2\mm-2s\right)^{s}}\cdot2^{-t}\binom{\sum_{i=1}^{t}(d_{i}-2)}{2r}\prod_{i=1}^{t}[d_{i}]_{2}\nonumber \\
 & \stackrel{\eqref{eq:two-ess}}{\le} & \frac{2^{2r}s!}{\left(\hd\nn-\hd\eps_{1}\nn\right)^{s}}\binom{kt}{2r}\prod_{i=1}^{t}[d_{i}]_{2}\nonumber \\
 & \le & \frac{2^{2r}s!}{\left((1-\eps_{1})\hd\nn\right)^{s}}\left(\frac{ekt}{2r}\right)^{2r}\prod_{i=1}^{t}[d_{i}]_{2}\nonumber \\
 & = & \frac{s!}{\left((1-\eps_{1})\hd\nn\right)^{s}}\left(\frac{ekt}{r}\right)^{2r}\prod_{i=1}^{t}[d_{i}]_{2}~.\label{eq:prob_admits}
\end{eqnarray}
For the random sequence $\dd$ of independent truncated Poissons under consideration, define $A_n(t,s)$ to be the expected number of bad 2-cores, that have $t$ vertices and $s$ edges, in the random $k$-core $\K$ (in the case the sequence is proper and has even sum --- otherwise treat the number as 0).  It then follows from above that for any given pair $(t,s)\in\cI_{\sprs}$ we have

\begin{eqnarray*}
A_n(t,s) & \le & \sum_{\substack{U\subset V(\K) \\ |U| = t}}\hspace{4pt}\sum_{\substack{d_{1},\ldots,d_{t}\ge k\\
\sum d_{i}\le kt+2r
}
}\Pr\left[\deg_{\K}(u_{i})=d_{i}\text{ for }i=1,2,\ldots,t\right]P(U,\ddt,s)\\
 & \stackrel{(\ref{eq:prob_admits})}\le & \binom{\nn}{t}\sum_{\substack{d_{1},\ldots,d_{t}\ge k\\
\sum d_{i}\le kt+2r
}
}\left(\prod_{i=1}^{t}\Pr\left[Z=d_{i}\right]\right)\cdot\frac{s!}{\left((1-\eps_{1})\hd\nn\right)^{s}}\left(\frac{ekt}{r}\right)^{2r}\prod_{i=1}^{t}[d_{i}]_{2}\\
 & \le & \frac{\nn^{t}}{t!}\frac{s!}{\left((1-\eps_{1})\hd\nn\right)^{s}} \left(\frac{ekt}{r}\right)^{2r}\sum_{\substack{d_{1},\ldots,d_{t}\ge k\\
\sum d_{i}\le kt+2r
}
}\prod_{i=1}^{t}[d_{i}]_{2}\Pr\left[Z=d_{i}\right]\\
 & \stackrel{\eqref{eq:longCalc}}{{\le}} & \frac{[s]_{r}}{\left((1-\eps_{1})\hd\right)^{s}\nn^{r}} \left(\frac{ekt}{r}\right)^{2r}\left((1-\delta)\hd\right)^{t}\left(\frac{2ek^{2}t}{r}\right)^{2r}\\
 & \le &  \left(\frac{1-\delta}{1-\eps_{1}}\right)^{t}\left(\frac{t}{r}\right)^{4r}\left(\frac{4se^{4}k^{6}}{(1-\eps_{1})\hd\nn}\right)^{r}\\
 & \stackrel{\eqref{eq:two-ess}}{\le} &  \left(\frac{1-\delta}{1-\eps_{1}}\left(\frac{t}{r}\right)^{4r/t}\right)^{t}\left(\frac{\eps_{1}2e^{4}k^{6}}{1-\eps_{1}}\right)^{r}.
\end{eqnarray*}
By definition of $\eps_{1}$ we have $\frac{1-\delta}{1-\eps_{1}}\le1-\delta/2$
and $\eps_{1}2e^{4}k^{6}\le1-\eps_{1}$. Using the fact that $x\mapsto x^{-x}$
is increasing for $0<x<1/e$, and by definition of $\delta_{1}$,
for every $(t,s)\in\cI_{\sprs}$ we obtain

\[
A_n(t,s)\le \left(\frac{1-\delta}{1-\eps_{1}}\left(\frac{r}{t}\right)^{-4r/t}\right)^{t}
\le\left(\left(1-\frac{\delta}{2}\right)\delta_{1}^{-4\delta_{1}}\right)^{t}\le\left(1-\frac{\delta}{4}\right)^{t}.
\]
Recalling that $t>\log^{2}n$, and observing that $\left|\mathcal{I}\right|<n^{3}$, we get
\[
\sum_{(t,s)\in\cI_{\sprs}}A_n(t,s)
\le \sum_{(t,s)\in\cI_{\sprs}}\left(1-\delta/4\right)^{t}
\le \left|\cI\right|\left(1-\delta/4\right)^{\log^{2}n}=O(1/n^2).
\]
The probability of the event $A_n$ is hence $O(1/n)$. Thus,  by Theorem~\ref{thm:maincor}, if $\dd$ is distributed as  the degree
sequence of $\K(n,c,k)$, the probability that a random core with degree sequence $\dd$ has a sparse 2-core and is proper  is $o(1)$. That is, WHP $\K(n,c,k)$ does not have both a proper degree sequence and a  sparse 2-core. Recalling that it has proper degree sequence WHP,   Lemma~\ref{lem:no-bad-two-cores} follows.

\end{document}